\documentclass[12pt]{article}
\usepackage{authblk}
\usepackage{graphicx} 
\usepackage{amsmath}
\usepackage{amsfonts}
\usepackage{amssymb}
\usepackage{amsthm}
\usepackage{float}
\usepackage{geometry}	
\usepackage[pdftex, backref=page]{hyperref} 
\renewcommand*{\backref}[1]{}
\renewcommand*{\backrefalt}[4]{%
	\ifcase #1 (Not cited.)%
	\or        (Cited on page~#2.)%
	\else      (Cited on pages~#2.)%
	\fi}

\usepackage[initials]{amsrefs}
\usepackage[english]{babel}

\usepackage{hyperref}
\usepackage{amssymb}
\usepackage{graphicx}
\usepackage{mathrsfs}
\usepackage{upref,amsthm,amsxtra,exscale}
\usepackage{cite}

\usepackage{caption}
\usepackage[shortlabels]{enumitem}
\usepackage{esint}
\setlist[description]{leftmargin=\parindent,labelindent=\parindent,itemsep=1pt,parsep=0pt,topsep=0pt}

\newtheorem{theorem}{Theorem}[section]
\newtheorem{corollary}[theorem]{Corollary}

\newtheorem{lemma}[theorem]{Lemma}

\newtheorem{conjecture}[theorem]{Conjecture}
\numberwithin{equation}{section}

\theoremstyle{definition}
\newtheorem{remark}[theorem]{Remark}

\usepackage{xcolor}

\def\R{\mathbb{R}}

\def\eps{\varepsilon}

\def\hat{\widehat}
\def\tilde{\widetilde}

\newcommand{\subscript}[2]{$#1 _ #2$}
\let\div\undefined
\DeclareMathOperator{\div}{div}

\newcommand{\mb}[1]{\mathbb{#1}}
\def\({\left(}
\def\){\right)}
\def\r{\mathbb{R}}
\def\r^n{\mathbb{R}^N}

\def\n{\mathbb{N}}

\def\eps{\varepsilon}

\def\hat{\widehat}
\def\tilde{\widetilde}

\def\l{\lambda}

\def\D{\Delta}

\def\p{\partial}
\def\n{\nabla}
\def\a{\alpha}

\def\cS{\mathcal{S}}

\title{Fractional De Giorgi conjecture in dimension 2 \\ via complex-plane methods}
\date{}

\author{Serena Dipierro,
	Jo{\~a}o Gon\c{c}alves da Silva,\\
	Giorgio Poggesi, and Enrico Valdinoci}

\affil{ {\footnotesize Department of Mathematics and Statistics,} 
	{\footnotesize The University of Western Australia,} \\
	{\footnotesize 35 Stirling Highway,
		Perth, WA 6009, Australia}\\
	{\footnotesize\tt serena.dipierro@uwa.edu.au},
	{\footnotesize\tt joao.goncalvesdasilva@research.uwa.edu.au},\\
	{\footnotesize\tt giorgio.poggesi@uwa.edu.au},
	{\footnotesize\tt enrico.valdinoci@uwa.edu.au}}

\begin{document}
	
	%
	\maketitle
	
	\begin{abstract}
		We provide a new proof of the fractional version of the De Giorgi conjecture for the Allen-Cahn equation in~$\mathbb{R}^2$ for the full range of exponents. Our proof combines a method introduced by A. Farina in~2003 with the $s$-harmonic extension of the fractional Laplacian
		in the half-space~$\mathbb{R}^{3}_+$ introduced by L. Caffarelli and L. Silvestre in~2007. 
		
		We also provide a representation formula for finite-energy weak solutions of a class of weighted elliptic partial differential equations in the half-space~$\mathbb{R}^{n+1}_+$ under Neumann boundary conditions. This generalizes the $s$-harmonic extension of the fractional Laplacian
		and allows us to relate a general problem in the extended space with a nonlocal problem on the trace.
	\end{abstract} 
	
	\section{Introduction}
	
	A major problem in the study of partial differential equations regards the classification of special solutions. 
	As a paradigmatic example, one can consider the case of the Allen-Cahn equation, describing the coexistence of two phases in a given material, and wonder whether ``natural'' solutions tend to segregate phases which are separated by a flat interface. A formal statement of this problem was formulated in a famous conjecture by Ennio De Giorgi, see~\cite{MR533166}:
	
	\begin{conjecture}\label{CDG}
		Let~$u\in C^2(\R^n, [-1, 1])$ be a solution of
		\[ \begin{cases}
			&-\Delta u(x)=u(x)-u^3(x),\\
			&\displaystyle\frac{\partial u}{\partial x_n}(x)>0,
		\end{cases}\]
		for all~$x\in\R^n$.
		
		Is it true that all the level sets of~$u$ are hyperplanes, at least when~$n \le 8$?
	\end{conjecture}
	
	Remarkably, Conjecture~\ref{CDG} is still open in its full generality, but a positive answer has been provided in dimension~$2$ and~$3$ by~\cite{MR1637919, MR1655510, MR1775735, MR1843784}, and a counterexample to flatness has been constructed in dimension~$9$ and higher in~\cite{MR2846486}.
	
	Conjecture~\ref{CDG} also holds true up to dimension~$7$ if the solution is a local minimizer of the associated energy functional, and also up to dimension~$8$ under the additional assumption that
	\begin{equation}\label{limcondition}
		\lim_{x_n\to\pm\infty} u(x_1,\dots,x_n)=\pm1,
	\end{equation}
	see~\cite{MR2480601}.
	
	Condition~\eqref{limcondition} can be further relaxed by supposing either that the profiles at infinity of the solution
	are two-dimensional, or that one level set is a complete graph, see~\cite{MR2728579}, or by taking suitable geometric assumptions on the limit interface of the solution, see~\cite{MR3488250}.
	
	When condition~\eqref{limcondition} is replaced by its uniform counterpart
	\begin{equation}\label{limcondition2}
		\lim_{M\to+\infty} \sup_{x_n>M} \big|u(x_1,\dots,x_n)- 1\big|+\sup_{x_n<-M} \big|u(x_1,\dots,x_n)+ 1\big|=0,
	\end{equation}
	then solutions are known to be one-dimensional in any spatial dimension~$n$, see~\cite{MR1755949, MR1763653, MR1765681}, and indeed the classification of the solutions of the Allen-Cahn equation under condition~\eqref{limcondition2} is a problem emerged in cosmology in relation to a conjecture by Gary William Gibbons. We refer to~\cite{MR2528756} for further information about the classical conjectures by De Giorgi and Gibbons.\medskip
	
	Suitable counterparts of Conjecture~\ref{CDG} have been considered in several different frameworks, including stratified groups~\cite{MR1949145, MR2461257, MR2579363, MR2773153}, hyperbolic spaces~\cite{MR2583503}, quasiminima~\cite{MR2413100, MR2728579}, nonlinear equations~\cite{MR1296785, MR1688549, MR1942128, MR2228294, MR2483642}, discrete settings~\cite{MR4224860},etc.\medskip
	
	Of great interest is also the fractional counterpart of Conjecture~\ref{CDG}, namely the situation in which the classical Laplace operator in the Allen-Cahn equation is replaced by a fractional Laplacian, defined, for every~$s\in(0,1)$, as
	$$ (-\Delta)^s u(x):=C_{n,s} \int_{\R^n}\frac{2u(x)-u(x+y)-u(x-y)}{|y|^{n+2s}}\,dy,$$
	where~$C_{n,s}>0$ is a suitable normalization constant.
	
	As a matter of fact, the fractional Allen-Cahn equation is also a very interesting topic of investigation in itself, being related to long-range interaction models and nonlocal minimal surfaces, see~\cite{MR4581189} and the references therein.
	
	The counterpart of Conjecture~\ref{CDG} for the fractional Laplacian~$(-\Delta)^s$ has been investigated in the fractional range~$s\in(0,1)$ by using different methods and under different perspectives. The first positive answer to the fractional version of Conjecture~\ref{CDG} was provided in~\cite{MR2177165} in dimension~$n=2$ for the fractional exponent~$s=\frac12$. The case~$n=2$ was then completed in~\cite{MR2498561, MR3280032} for the full range of the fractional exponents~$s\in(0,1)$.
	
	The fractional case of Conjecture~\ref{CDG} in dimension~$n=3$ was first addressed in~\cite{MR2644786} when~$s=\frac12$ and then extended to the case~$s\in\left(\frac12,1\right)$ in~\cite{MR3148114}. The three-dimensional picture was thus completed, by different methods, in~\cite{MR4124116, MR3740395}, which establish the validity of the fractional counterpart of Conjecture~\ref{CDG} when~$n=3$ and~$s\in\left(0,\frac12\right)$.
	
	So far, in dimension~$n=4$ the only case in which the fractional version of Conjecture~\ref{CDG} has been proved is that of~$s=\frac{1}{2}$, see~\cite{MR4050103}. 
	
	The classification of fractional solutions
	under the limit condition in~\eqref{limcondition} has been established in~\cite{MR3812860} when~$n\le8$ and~$s=\frac{1}{2}$, in~\cite{MR3939768} when~$n\le8$ and~$s\in\left( \frac{1}{2},1\right)$, and in~\cite{MR4124116} when either~$n\le3$ and~$s\in\left(0,\frac{1}{2}\right)$, or~$n\le8$ and~$s\in\left(\frac{1}{2}-\eta_n, \frac{1}{2}\right)$
	for some~$\eta_n\in\left(0, \frac{1}{2}\right)$ (i.e., when the fractional exponent is below the threshold~$\frac{1}{2}$, but not too far from it).
	
	The fractional counterpart of the conjecture by Gibbons has also been established in all spatial dimensions in~\cite{MR2952412},
	see also~\cite{MR3280032}, but all the remaining cases are still wide open, and, to the best of our knowledge, no counterexample is yet available in the literature.\medskip
	
	In this paper, we provide a new proof of the fractional version of Conjecture~\ref{CDG} in dimension~$n=2$, for the full range of fractional exponents~$s\in(0,1)$. The proof that we present here utilizes an elegant method invented in~\cite{MR2014827} which conveniently exploits a complex formulation of the equation. In this framework, we have:
	
	\begin{theorem}\label{Theorem 1}
		Let~$f \in C^1(\mathbb{R})$ and~$u\in C^2(\mathbb{R}^2)$ be a solution of
		\begin{equation}\label{Problem 1}
			\begin{cases}
				(-\Delta)^s u = f(u) & \;\text{ in }\; \mathbb{R}^2,\\
				\partial_{x_2} u\geq 0.
			\end{cases}
		\end{equation}
		
		Assume that, for all~$x\in \mathbb{R}^2$ and~$t>0$,
		\begin{equation}\label{decay assumption on u}
			\int_{\mathbb{R}^2}\frac{|u(x-ty)|}{(1+|y|^2)^{1+s}}\,dy<+\infty
		\end{equation}and
		\begin{equation}\label{decay assumption on u_x_2}
			\int_{\mathbb{R}^2} \frac{\partial_{x_2}u(x-ty)}{(1+|y|^2)^{1+s}}\,dy<+\infty.
		\end{equation}
		
		Then, $u$ is one-dimensional, i.e., there exist~$u_0:\mathbb{R}\to \mathbb{R}$ and~$\omega \in\mb{S}^1$ such that~$u(x) = u_0(\omega\cdot x)$.
	\end{theorem}
	
	We remark that condition~\eqref{decay assumption on u} 
	is fulfilled if~$u\in L^\infty(\mathbb{R}^2)$. Actually, bounded solutions of~\eqref{Problem 1}
	also satisfy~\eqref{decay assumption on u_x_2}, due to a bootstrap application of
	the fractional elliptic regularity theory, see e.g.~\cite{zbMATH07813655}, therefore
	Theorem~\ref{Theorem 1} can be seen as a general version of the fractional counterpart of
	Conjecture~\ref{CDG}, thus recovering the main results on this topic
	in~\cite{MR2498561, MR3280032}.\medskip
	
	A common treat between~\cite{MR2498561, MR3280032} and the approach presented here is that, to prove Theorem~\ref{Theorem 1}, we will make use of the $s$-harmonic extension\footnote{See~\cite[Section~5.2]{MR3469920} and~\cite{MR4124116} for proofs of fractional versions of the De Giorgi Conjecture that do not hinge on extension methods.}
	of the fractional Laplacian to the upper half plane~${\mathbb{R}}_{+}^3:=\mathbb{R}^2\times(0,+\infty)$
	(see~\cite{MR247668, MR2354493} and the forthcoming Section~\ref{PF:TH1}).
	For this, points in~$ \mathbb{R}_{+}^3$ will be denoted by~$(x_1,x_2,t)\in\mathbb{R}^2\times(0,+\infty)$.
	
	This approach will lead us to study a rather general weighted equation in the half-space
	with Neumann-type boundary conditions, see~\cite{MR2560300} for related results. 
	In fact, we will deduce Theorem~\ref{Theorem 1} from the following result:
	
	\begin{theorem}\label{Main Theorem}
		Let~$a \in L_{\text{loc}}^1([0,+\infty))$ be continuously differentiable, $F$, $f \in C^1(\mathbb{R})$ and~$U\in C^2_{\text{loc}}(\overline{\mathbb{R}_{+}^3})$ be a solution of 
		\begin{equation}\label{Main Problem}
			\begin{cases}
				-\div(a(t)\nabla U) = F(U) & \;\text{ in }\; \mathbb{R}_{+}^3,\\
				-\displaystyle\lim_{t\to0^+}a(t)\,\partial_t U = f(U) & \;\text{ in }\; \mathbb{R}^2\times \{0\},\\
				\partial_{x_2} U > 0 &\;\text{ in }\; \mathbb{R}_{+}^{3},
			\end{cases}
		\end{equation}
		such that 
		\begin{equation}\label{sufficien condition}
			\sup_{R>1}\frac{1}{R^2}\int_{B_{2R}^+\setminus B_{R}^+ }a(t)\,(\partial_{x_2}U(x,t))^2 \,dx \,dt < +\infty.
		\end{equation}
		
		Then, for each~$t\geq 0$, $U(\cdot, t)$ is one-dimensional, i.e., for each~$t\geq 0$, there exist~$u_t: \mathbb{R}\rightarrow \mathbb{R}$ and~$\omega_t \in \mathbb{S}^1$ such that~$U(x,t) = u_t(x\cdot \omega_t)$ for every~$x \in \mathbb{R}^2$.
	\end{theorem}
	
	Note that when~$a(t):= t^{1-2s}$ and~$F\equiv 0$, the trace~$u$ of a solution of~\eqref{Main Problem} to~$\mathbb{R}^2\times \{0\}$ is, up to a normalizing constant, a one-dimensional solution of the nonlocal equation~$(-\Delta)^s u = f(u)$ (in~$\mathbb{R}^2$).
	
	Our strategy to prove Theorem~\ref{Theorem 1} is as follows. We consider a solution of~\eqref{Problem 1} satisfying~\eqref{decay assumption on u} and~\eqref{decay assumption on u_x_2}. We then use the $s$-harmonic extension of the fractional Laplacian to obtain a solution of~\eqref{Main Problem} that also satisfies the boundedness condition~\eqref{sufficien condition} (due to assumptions~\eqref{decay assumption on u} and~\eqref{decay assumption on u_x_2}). Thus, by applying Theorem~\ref{Main Theorem}, we conclude that the solution must be one-dimensional.
	\medskip


In this paper, we also obtain a representation formula for finite-energy solutions of a class of weighted equations in the half-space~$\mathbb{R}^{n+1}_+:= \mathbb{R}^n\times (0,+\infty)$ involving Neumann boundary conditions.
This representation formula, which is contained in Theorem~\ref{Convolution enters the chat}, extends the $s$-harmonic extension to the half-space of the fractional Laplacian of Caffarelli and Silvestre in~\cite{MR2354493}
to a larger class of equations.
	
	We employ results from~\cite{MR2354493, MR3233760, MR4204715} to obtain such a formula, including the finite-energy Liouville Theorem in~\cite{MR4204715}, the Fourier transform approach to the $s$-harmonic extension of the fractional Laplacian to the half-space in~\cite{MR2354493}, and the in-depth analysis in~\cite{MR3233760} of a one-dimensional variational problem that arises from the energy functional associated with the equations we will be dealing with. 
	
In our setting, this approach is helpful in two ways: first, it allows us to relate a general problem in the extended space to a nonlocal problem on the trace; second, it enables us to reduce general nonlocal problems on the trace to extended problems of a local type.

	The extension method for nonlocal equations has been widely used and extended, see e.g.~\cite{MR2754080, MR3056307, MR3709888}.
	Among the different possible approaches, see in particular~\cite{MR2754080}, in which a general operator in the space variables is allowed in the extended formulation; our setting however does not fall into this framework, since the weight in the extended variable is not necessarily of monomial type (also, as a technical remark, we rely here on Fourier analysis more than semigroup theory).

Furthermore, the work developed in Section \ref{section4} contains several results of general use which may be useful in future applications, including growth bounds and integrability properties of ~$A_2$-Muckenhoupt weights. We refer to Section~\ref{section4} for details.
	 
	\medskip
	
	The rest of the paper is organized as follows.
	In Section~\ref{section3} we prove Theorem~\ref{Main Theorem}. In Section~\ref{PF:TH1}, we employ Theorem~\ref{Main Theorem} to establish Theorem~\ref{Theorem 1}. In Section~\ref{section4} we obtain a representation formula for finite-energy weak solutions of a class of weighted equations in~$\mathbb{R}^{n+1}_+$. The paper ends with Appendix~\ref{gyu4io3vbcnx76859430}, where we collect the proofs of
	some auxiliary lemmata.
	
	\section{Proof of Theorem~\ref{Main Theorem}}\label{section3}
	
	To prove Theorem~\ref{Main Theorem} it is useful to combine
	the complex variable notation in the plane with the cylindrical coordinates in~${\mathbb{R}}^3$.
	To this end, in the notation of Theorem~\ref{Main Theorem},
	we define the auxiliary functions 
	\begin{equation}\label{Z, theta, rho}\begin{split}
			&    Z(x,t) := \partial_{x_1}U(x,t)+i \partial_{x_2}U(x,t), \qquad \rho(x,t):= \sqrt{(\partial_{x_1}U)^2+(\partial_{x_2}U)^2}\\&{\mbox{and}}\qquad
			\theta(x,t) := \arcsin \left(\frac{\partial_{x_2}U}{\rho}\right)
	\end{split}\end{equation}
	and we note that~$Z = \rho e^{i\theta}$.
	
	The usefulness of these functions lies in the following observation:
	
	\begin{lemma}\label{lmminopiccolino}
		Let~$a \in L_{\text{loc}}^1([0,+\infty))$ be continuously differentiable, $F$, $f \in C^1(\mathbb{R})$ and~$U\in C^2_{\text{loc}}(\overline{\mathbb{R}_{+}^3})$ be a solution of 
		\begin{equation}\label{Main ProblemBIS}
			\begin{cases}
				-\div(a(t)\nabla U) = F(U) & \;\text{ in }\; \mathbb{R}_{+}^3,\\
				-\displaystyle\lim_{t\to0^+}a(t)\,\partial_t U = f(U) & \;\text{ in }\; \mathbb{R}^2\times \{0\}.
			\end{cases}
		\end{equation}
		
		Then, in the notation of~\eqref{Z, theta, rho},
		\begin{equation}\label{eq for theta}
			\begin{cases}
				\div (a(t)\rho^2 \nabla \theta) = 0&\; \text{ in }\;{\mathbb{R}}_{+}^3,\\
				\displaystyle  -\lim_{t\to 0^+}a(t)\rho\,\partial_t \theta = 0&\; \text{ in } \;{\mathbb{R}}^2\times \{0\}.
			\end{cases}
		\end{equation}
	\end{lemma}
	
	\begin{proof}
		{F}rom~\eqref{Main ProblemBIS} and the fact that~$Z=\partial_{x_1}U+i \partial_{x_2}U$, we see that~$Z$ is a solution of \begin{equation}\label{Eq for Z}
			\begin{cases}
				\div(a(t)\nabla Z)= F'(U)Z&\; \text{ in }\;{\mathbb{R}}_{+}^3,\\
				\displaystyle
				-\lim_{t\to0^+}a(t)\partial_t Z = f'(U)Z&\; \text{ in } \;{\mathbb{R}}^2\times \{0\}.
			\end{cases}
		\end{equation}
		Since
		$$\Delta Z = \Big(\Delta \rho -\rho|\nabla \theta|^2+2i\nabla \rho \cdot \nabla \theta + i\rho \Delta \theta\Big)e^{i\theta},$$
		we deduce from~\eqref{Eq for Z} that
		\begin{equation*}\label{ext eq for theta}
			\begin{cases}
				a(t)\big(\Delta \rho -\rho|\nabla \theta|^2+2i\nabla \rho \cdot \nabla \theta + i\rho \Delta \theta\big)
				+ a'(t)\big(\partial_t \rho + i\rho \partial_t \theta\big) = \rho F'(U) &\; \text{ in }{\mathbb{R}}_{+}^3,\\
				\displaystyle-\lim_{t\to0^+}a(t)\big(\partial_t \rho + i\rho\partial_t \theta\big)= \rho f'(U)&\; \text{ in }\;{\mathbb{R}}^2\times \{0\}.
			\end{cases}
		\end{equation*}Hence,
		taking the imaginary part,
		\begin{equation*}
			\begin{cases}
				a(t)\big(2\n \rho \cdot \n \theta + \rho \D \theta\big)+ a'(t)\rho \partial_t\theta = 0&\; \text{ in } \;{\mathbb{R}}_{+}^3,\\
				\displaystyle-\lim_{t\to0^+}
				a(t)\rho\,\partial_t \theta = 0&\; \text{ in }\;{\mathbb{R}}^2\times \{0\},
			\end{cases}
		\end{equation*}
		which, after a multiplication by~$\rho$ of the first equation, can be rewritten as~\eqref{eq for theta}.
	\end{proof}
	
	As remarked in~\cite{MR1655510, MR2014827}, equations in divergence form
	such as~\eqref{eq for theta} naturally produce Liouville-type results after testing the equation against appropriate bump functions. This idea can be adapted to our framework, leading to the following result:
	
	\begin{lemma}\label{LIOUV}
		Let~$a \in L_{\text{loc}}^1([0,+\infty))$  be continuously differentiable, $F$, $f \in C^1(\mathbb{R})$ and~$U\in C^2_{\text{loc}}(\overline{\mathbb{R}_{+}^3})$ be a solution of~\eqref{Main Problem}
		satisfying~\eqref{sufficien condition}.
		Let~$\theta$ be as in~\eqref{Z, theta, rho}.
		
		Then, $\theta$ is constant in~$\overline{\mathbb{R}_{+}^3}$.
	\end{lemma}
	
	\begin{proof}
		To show the desired result we employ Lemma~\ref{lmminopiccolino} and we integrate~\eqref{eq for theta} against a suitable test function. For this, let~$\tau : {\mathbb{R}}_{+}^{3}\to {\mathbb{R}}$ be a compactly supported smooth function such that
		\begin{equation}\label{zzz}
			\text{supp}(\tau) \subset {B_{2R}^+}, \qquad \tau\big|_{{{B_{R}^{+}}}} \equiv 1\qquad {\mbox{and}}\qquad |\n \tau|\leq \frac{C}{R}, 
		\end{equation}
		for some positive constant~$C$. Here~$B_R^+ := B_R \cap {\mathbb{R}}_{+}^3$. 
		
		Then, integrating~\eqref{eq for theta} against the test function~$\tau^2\theta$, we have that
		\begin{equation}\label{ibp}
			\begin{split}
				0=& \int_{{\mathbb{R}}_{+}^3} \div\big(a(t)\rho^2\nabla \theta\big)\tau^2\theta \,dx \,dt \\
				= &-\int_{{\mathbb{R}}_{+}^3} a(t)\rho^2\nabla\theta \cdot \nabla(\tau^2\theta) \,dx \,dt - \int_{B_{2R}\cap \{t=0\}}a(t)\tau^2\rho^2 \theta \partial_t \theta\, dx\\
				=& -\int_{{\mathbb{R}}_{+}^3} a(t)\rho^2 \tau^2|\nabla\theta|^2 \,dx \,dt - 2\int_{{\mathbb{R}}_{+}^3} a(t)\rho^2 \tau \theta\nabla\theta\cdot \nabla \tau \,dx \,dt.
			\end{split}
		\end{equation}
		We point out that in the last line above we have used the second equation in~\eqref{eq for theta}.
		
		Now, using~\eqref{zzz} and H\"older's inequality we see that
		\begin{equation}\label{f849368tghudsiety984tgfueiw76yr39843ytfuhgeskuf}
			\begin{split}
				& \left| \int_{{\mathbb{R}}_{+}^3}a(t)\rho^2 \tau \theta\nabla\theta\cdot \nabla \tau  \,dx \,dt\right|
				=\left|\int_{B_{2R}^+\setminus B_R^+} a(t)\rho^2 \tau \theta\nabla\theta\cdot \nabla \tau \,dx \,dt\right| \\
				&\qquad\leq \left(\int_{B_{2R}^+\setminus B_R^+} a(t)\rho^2 \tau^2 |\nabla \theta|^2 \,dx \,dt\right)^{1/2} \left(\int_{B_{2R}^+\setminus B_R^+} a(t)\rho^2\theta^2|\n\tau|^2  \,dx \,dt\right)^{1/2}.
			\end{split}
		\end{equation}
		Thanks to the last equation in~\eqref{Main Problem}, we also observe that
		$$ \theta =\arcsin \left(\frac{\partial_{x_2}U}{\rho}\right)\leq \frac{C_1 \partial_{x_2}U}\rho,
		$$ for some~$C_1>0$.
		
		Plugging this information into~\eqref{f849368tghudsiety984tgfueiw76yr39843ytfuhgeskuf}, and recalling~\eqref{zzz}, we get that
		\begin{equation*}
			\begin{split}
				&  \left|\int_{{\mathbb{R}}_{+}^3}a(t)\rho^2 \tau \theta\nabla\theta\cdot \nabla \tau  \,dx \,dt\right|
				\\ &\qquad\leq C\left(\int_{B_{2R}^+\setminus B_R^+} a(t)\rho^2 \tau^2 |\nabla \theta|^2 \,dx \,dt\right)^{1/2} 
				\left(\frac{1}{R^2}\int_{B_{2R}^+\setminus B_R^+} a(t)\rho^2\left(\frac{C_1 \partial_{x_2}U}\rho\right)^2\,dx \,dt\right)^{1/2}\\
				&\qquad= CC_1\left(\int_{B_{2R}^+\setminus B_R^+} a(t)\rho^2 \tau^2 |\nabla \theta|^2 \,dx \,dt\right)^{1/2} 
				\left(\frac{1}{R^2}\int_{B_{2R}^+\setminus B_R^+} a(t)\left( \partial_{x_2}U\right)^2\,dx \,dt\right)^{1/2}  .
			\end{split}
		\end{equation*}
		Thus, exploiting the assumption in~\eqref{sufficien condition}, we find that 
		\begin{equation*}\left|
			\int_{{\mathbb{R}}_{+}^3} a(t)\rho^2 \tau \theta\nabla\theta\cdot \nabla \tau \,dx \,dt\right|
			\leq C_2\left(\int_{B_{2R}^+\setminus B_R^+} a(t)\rho^2\tau^2 |\n \theta|^2 \,dx \,dt\right)^{1/2},
		\end{equation*}
		for some~$C_2>0$ independent of~$R$.
		
		{F}rom this and~\eqref{ibp}, we conclude that
		\begin{equation}\label{nbvceuity47ie4356728}\begin{split}&
				\int_{B_{2R}^+} a(t)\rho^2 \tau^2|\nabla\theta|^2 \,dx \,dt=
				\int_{{\mathbb{R}}_{+}^3} a(t)\rho^2 \tau^2|\nabla\theta|^2 \,dx \,dt\\&\qquad=-2
				\int_{{\mathbb{R}}_{+}^3} a(t)\rho^2 \tau \theta\nabla\theta\cdot \nabla \tau \,dx \,dt
				\le 2C_2\left(\int_{B_{2R}^+\setminus B_R^+} a(t)\rho^2\tau^2 |\n \theta|^2 \,dx \,dt\right)^{1/2},
		\end{split}\end{equation}
		and therefore
		\begin{equation*}
			\left(\int_{B_{2R}^+} a(t)\rho^2\tau^2 |\n\theta|^2 \,dx \,dt\right)^{1/2}\leq2 C_2.
		\end{equation*}
		As a result,
		\begin{equation*}
			\int_{B_{R}^+} a(t)\rho^2 |\n\theta|^2 \,dx \,dt\leq 4 C_2^2.
		\end{equation*}
		Thus, sending~$R$ to infinity, we get that
		\begin{equation*}
			\int_{{\mathbb{R}}_{+}^3} a(t)\rho^2 |\n\theta|^2 \,dx \,dt\leq 4C_2^2.
		\end{equation*}
		
		As a consequence, 
		we have that~$a(t)\rho^2 |\n \theta|^2\in L^1({\mathbb{R}}_{+}^3)$.
		We now exploit this information together with~\eqref{nbvceuity47ie4356728}
		to see that
		\begin{equation*}
			\begin{split}
				\int_{{\mathbb{R}}_{+}^3} a(t)\rho^2 |\n\theta|^2 \,dx \,dt&
				= \lim_{R\to +\infty} \int_{B_{R}^+} a(t)\rho^2 |\n\theta|^2 \,dx \,dt\\
				&\leq 2C_2 \lim_{R\to +\infty} \left(\int_{B_{2R}^+\setminus B_R^+} a(t)\rho^2\tau^2 |\n \theta|^2 \,dx \,dt\right)^{1/2}\\
				&\leq 2C_2 \lim_{R\to +\infty} \left(\int_{B_{2R}^+\setminus B_R^+} a(t)\rho^2 |\n \theta|^2 \,dx \,dt\right)^{1/2}\\
				&= 0.
			\end{split}
		\end{equation*}
		Since~$\rho\neq0$ (thanks to the last equation in~\eqref{Main Problem}), this
		gives that~$\n \theta = 0$ in~${\mathbb{R}}_+^3$, and therefore~$\theta$
		is constant in~${\mathbb{R}}_+^3$.
		The desired result follows from this and the continuity of~$\theta$.
	\end{proof}
	
	\begin{proof}[Proof of Theorem~\ref{Main Theorem}]
		We exploit the notation in~\eqref{Z, theta, rho} and we use 
		Lemma~\ref{LIOUV} to obtain that~$\theta \equiv \theta_0$ 
		in~$\overline{\mathbb{R}_{+}^3}$,
		for some~$\theta_0\in {\mathbb{R}}$. 
		As a consequence,
		we can write~$(\partial_{x_2}U)^2= \alpha \rho^2$, for some~$\a\in (0,1]$.
		{F}rom this, we also have that~$\partial_{x_1}U = \pm \sqrt{\frac{1-\a}\alpha} \partial_{x_2}U$. 
		This yields that there exists~$\nu\in \mathbb{R}^2\setminus\{0\}$,
		possibly depending on~$t$, such that~$\nabla U = (\nu \partial_{x_2}U,\partial_t U)$.
		
		Let now~$\nu_{\perp}\in {\mathbb{R}}^2$ be the unit vector orthogonal to~$\nu$. Then, for all~$ \eta\in \mathbb{R}$ and~$ \nu_0\in\mathbb{R}^2$,
		\begin{eqnarray*}&&
			\frac{d}{d\eta}U(\eta\nu_{\perp}+\nu_0, t) =\big(\partial_{x_1}U(\eta\nu_{\perp}+\nu_0, t),
			\partial_{x_2}U(\eta\nu_{\perp}+\nu_0, t)\big)\cdot\nu_\perp \\&&\qquad\qquad=
			\partial_{x_2}U(\eta\nu_{\perp}+\nu_0, t)\,\nu\cdot\nu_\perp
			= 0,
		\end{eqnarray*}
		that is, $U(\cdot, t)$ is constant along straight lines parallel to~$\nu_{\perp}$. More precisely, for all~$ \eta\in \mathbb{R}$ and~$ \nu_0\in\mathbb{R}^2$,
		\begin{equation}\label{dweiotgurgbvhsdvt75i43wmnbvc}
			U(\eta\nu_{\perp}+\nu_0, t)=U(\nu_0, t).
		\end{equation}
		
		Now, for all~$t\geq 0$, we define $$u_t(\eta):= U\left(\frac{\eta\nu}{|\nu|},t\right).$$
		In light of~\eqref{dweiotgurgbvhsdvt75i43wmnbvc}, we have that
		\begin{eqnarray*}&&U(x,t) = 
			U\left( \left(\frac{\nu}{|\nu|}\cdot x\right)\frac{\nu}{|\nu|}+ \left(\frac{\nu_\perp}{|\nu_\perp|}\cdot x\right) \frac{\nu_\perp}{|\nu_\perp|},t
			\right)=U\left( \left(\frac{\nu}{|\nu|}\cdot x\right)\frac{\nu}{|\nu|},t
			\right)=
			u_t\left(\frac{\nu}{|\nu|}\cdot x\right),\end{eqnarray*}
		which establishes the desired result.
	\end{proof}
	
	\section{\textbf{Fractional Version of the De Giorgi Conjecture in} $\mathbb{R}^2$ \textbf{and proof of Theorem~\ref{Theorem 1}}}\label{PF:TH1}
	
	Here we will exploit the $s$-harmonic extension of the fractional Laplacian to the upper-half space~$\mathbb{R}_{+}^3$ to prove Theorem~\ref{Theorem 1}. 
	To fulfill this goal, 
	as done in~\cite{MR2354493, MR2498561}, 
	we consider the
	$s$-Poisson kernel associated with the fractional Laplacian, defined,
	for all~$(x,t)\in\R^3_+$, as
	$$P_s(x,t):= \frac{c_{n,s}\,t^{2s}}{\left(t^2+|x|^2\right)^{\frac{n+2s}{2}}}.$$
	Here~$c_{n,s}$ is a normalizing constant such that, for all~$ t>0$,
	\begin{equation*}
		\int_{{\mathbb{R}}^2}P_{s}(x,t)\,dx=1.
	\end{equation*}
	
	Furthermore, for all~$(x,t)\in\R^3_+$, we define
	\begin{equation}\label{Sol of extension}
		U(x,t) := \int_{{\mathbb{R}}^2}P_s(y,t)u(x-y)\,dy.
	\end{equation}
	We have that~$U(x,0)=u(x)$, see e.g. Lemma~12 in~\cite{MR2498561}.
	
	Also, we know that if~$u$ solves~$(-\Delta)^su=f(u)$ in~$\R^2$, then~$U$ satisfies the following problem
	\begin{equation}\label{extension}
		\begin{cases}
			\div (t^{1-2s}\n U )= 0 &\; \text{ in } \; {\mathbb{R}}_{+}^3,\\
			-\displaystyle\lim_{t\to 0^+} t^{1-2s}\partial_t  U = f(u)&\; \text{ in } \;{\mathbb{R}}^2\times \{0\}.
		\end{cases}
	\end{equation}
	In this setting, we have the following result:
	
	\begin{theorem}\label{dhuityintermediate}
		Let~$f\in C^1(\R)$ and~$u$ be a solution of~\eqref{Problem 1}
		satisfying~\eqref{decay assumption on u}
		and~\eqref{decay assumption on u_x_2}. Let~$U$ be as in~\eqref{Sol of extension}.
		
		Then, there exist~$\omega\in\mathbb{S}^1$ and,
		for all~$t\geq 0$, a function~$u_t: \mathbb{R}\rightarrow \mathbb{R}$
		such that~$U(x,t) = u_t(x\cdot \omega)$ for every~$x \in \mathbb{R}^2$.
	\end{theorem}
	
	The proof of Theorem~\ref{dhuityintermediate}
	will rely on Theorem~\ref{Main Theorem}. 
	For this, we will now check that the assumptions
	of Theorem~\ref{Main Theorem} are satisfied.
	
	We first point out that 
	Theorem~\ref{dhuityintermediate} (and also Theorem~\ref{Theorem 1}) is trivial if~$\partial_{x_2}u\equiv 0$, since, by~\eqref{Sol of extension},
	$$ \partial_{x_2} U(x,t) = \int_{{\mathbb{R}}^2}P_s(y,t)\partial_{x_2}u(x-y)\,dy=0,$$
	for all~$(x,t)\in\R^3_+$.
	Therefore, from now on, we assume that
	\begin{equation}\label{nvcxba123456709876}
		{\mbox{$\partial_{x_2}u(x_0)>0$ at some point~$x_0\in {\mathbb{R}}^2$.}}\end{equation}
	
	Also, we observe that~$\partial_{x_2}U\ge0$, thanks to~\eqref{Sol of extension}
	and the second equation in~\eqref{Problem 1}.
	Moreover, in light of~\eqref{nvcxba123456709876} and the continuity of~$u$,
	we actually have that
	\begin{equation}\label{f543q45367289dhvhdbcdhsj}
		{\mbox{$\partial_{x_2}U>0$ in~$\R^3_+$.}}\end{equation}
	
	Now, we show that~$U$, as defined in~\eqref{Sol of extension}, satisfies~\eqref{sufficien condition}. 
	To this end, we provide a preliminary pointwise estimate on~$\partial_{x_2}U$.
	
	\begin{lemma}\label{est U_y}
		Under the assumptions of Theorem~\ref{dhuityintermediate}, there exists a positive constant~$C_0$, depending on~$n$, $s$ and~$u$, such that, for all~$ (x,t)\in {\mathbb{R}}_{+}^3$,
		$$|\partial_{x_2}U(x,t)|\leq C_0 \min\{1, t^{-1}\}.$$
	\end{lemma}
	
	\begin{proof}
		Using the assumption in~\eqref{decay assumption on u_x_2} and the change of variables~$w:= t^{-1}y$, we see that
		\begin{equation*}
			\partial_{x_2}U(x,t)= c_{n,s}\int_{{\mathbb{R}}^2} \frac{t^{2s}\,\partial_{x_2}u(x-y)}{(t^2+|y|^2)^{1+s}}\,dy=c_{n,s} \int_{{\mathbb{R}}^2}\frac{\partial_{x_2}u(x-tw)}{(1+|w|^2)^{1+s}}\,dw <+\infty.
		\end{equation*}
		Moreover, integrating by parts and changing variables~$w := t^{-1}y$, we have that
		\begin{equation*}
			\begin{split}
				\partial_{x_2}U (x,t) &= \int_{{\mathbb{R}}^2}P_{s}(y,t) \partial_{x_2}u(x-y)\,dy\\& = \int_{{\mathbb{R}}^2}\p_{y_2}P_s(y,t)u(x-y)\,dy\\
				&= 2(1+s)c_{n,s}\int_{{\mathbb{R}}^2}\frac{t^{2s}y_2}{(t^2+|y|^2)^{2+s}}u(x-y)\,dy\\
				&= 2(1+s)c_{n,s}t^{2s+3-4-2s} \int_{{\mathbb{R}}^2} \frac{w_2}{(1+|w|^2)^{2+s}}u(x-tw)\,dw\\
				&\leq2(1+s)c_{n,s}t^{-1}\int_{{\mathbb{R}}^2} \frac{|u(x-tw)|}{(1+|w|^2)^{1+s}}\,dw\\
				&\leq 2(1+s)c_{n,s}Ct^{-1},
			\end{split}
		\end{equation*}
		where in the last inequality we have used the assumption in~\eqref{decay assumption on u}.
		
		The desired result now follows from the last two formulas in display and~\eqref{f543q45367289dhvhdbcdhsj}. 
	\end{proof}
	
	\begin{lemma}\label{Sufficient condition is satisfies}
		Under the assumptions of Theorem~\ref{dhuityintermediate}, we have that~$U$, as defined in~\eqref{Sol of extension}, satisfies~\eqref{sufficien condition}
		with~$a(t):=t^{1-2s}$.
	\end{lemma}
	
	\begin{proof}
		Let~$R>1$ and~$C_{R} := B_{R}\times (0,R) \subset \mathbb{R}_{+}^3$ denote the cylinder of radius~$R$ and height~$R$ in the upper half-space~$\mathbb{R}_{+}^3$. {F}rom Lemma~\ref{est U_y} we have that 
		\begin{equation*}
			\begin{split}
				\frac1{R^2} \int_{B_{2R}^+\setminus B_{2R}^+} t^{1-2s}(\partial_{x_2}U)^2\,dx \,dt& \leq \frac{C}{R^2} \int_{B_{2R}^+\setminus B_{2R}^+} t^{1-2s}\min\{1,t^{-2}\} \,dx \,dt\\
				&\leq \frac{C}{R^2} \int_{C_{2R}} t^{1-2s}\min\{1,t^{-2}\} \,dx \,dt\\
				&\leq C\int_{0}^{+\infty} t^{1-2s}\min\{1, t^{-2}\}\,dt \\&\leq C,
			\end{split}
		\end{equation*}
		up to renaming~$C>0$,
		that is independent of~$R$.
	\end{proof}
	
	We can now complete the proof of Theorem~\ref{dhuityintermediate}.
	
	\begin{proof}[Proof of Theorem~\ref{dhuityintermediate}]
		We have that~$U$ satisfies~\eqref{extension} and~\eqref{f543q45367289dhvhdbcdhsj},
		and therefore the assumptions in~\eqref{Main Problem} are fulfilled taking~$a(t):=t^{1-2s}$.
		
		Also, $U$ satisfies~\eqref{sufficien condition} with~$a(t)=t^{1-2s}$, thanks to
		Lemma~\ref{Sufficient condition is satisfies}.
		
		Therefore, we are in the position of exploiting Theorem~\ref{Main Theorem}
		in this setting. In this way, we obtain that, for all~$t\ge0$, there exist~$u_t: \mathbb{R}\rightarrow \mathbb{R}$ and~$\omega_t \in \mathbb{S}^1$ such that
		\begin{equation}\label{zxcvbnmoiuytUYLIHLIOIOI59043}
			U(x,t) = u_t(x\cdot \omega_t) \end{equation} for every~$x \in \mathbb{R}^2$.
		
		Now, to complete the proof of Theorem~\ref{dhuityintermediate},
		we show that
		\begin{equation}\label{65743982gfufiyefjhkuewthzxcvbn}
			{\mbox{there exists~$\omega\in \mathbb{S}^1$ such that~$\omega_t=\omega$
					for all~$t\ge0$.}}
		\end{equation}
		To check this, we use~\eqref{zxcvbnmoiuytUYLIHLIOIOI59043} with~$t=0$ and we have that~$u(x)=U(x,0)=u_0(x\cdot\omega_0)$.
		We let~$(\omega_0)_\perp\in\R^2$ be the unit vector orthogonal to~$\omega_0$.
		Therefore, recalling~\eqref{Sol of extension}, we have that, for all~$t>0$ and~$\eta\in\R$,
		\begin{eqnarray*} U(x+\eta(\omega_0)_\perp,t) &=& \int_{{\mathbb{R}}^2}P_s(y,t)u(x+\eta(\omega_0)_\perp-y)\,dy\\&=&
			\int_{{\mathbb{R}}^2}P_s(y,t)u_0\big((x+\eta(\omega_0)_\perp-y)\cdot\omega_0\big)\,dy\\&=&
			\int_{{\mathbb{R}}^2}P_s(y,t)u_0\big((x-y)\cdot\omega_0\big)\,dy\\&=&
			\int_{{\mathbb{R}}^2}P_s(y,t)u(x-y)\,dy\\&=&U(x,t)
			.
		\end{eqnarray*}
		That is, for all~$t>0$, we have that~$U(\cdot, t)$ is constant along straight lines parallel to~$(\omega_0)_\perp$.
		
		Accordingly, \eqref{65743982gfufiyefjhkuewthzxcvbn} is satisfied with~$\omega:=\omega_0$, and this completes the proof of Theorem~\ref{dhuityintermediate}.
	\end{proof}
	
	Theorem~\ref{Theorem 1} now plainly follows from Theorem~\ref{dhuityintermediate}, since~$U(x,0)=u(x)$.
	
	\section{Representation Formula for a Class of Weighted Equations in~$\mathbb{R}^{n+1}_+$}\label{section4}
	
	In this section, we present a representation formula for finite-energy solutions of a class of weighted elliptic partial differential equations in the upper-half space~${\mathbb{R}}^{n+1}_+:={\mathbb{R}}^n\times(0,+\infty)$, with~$n\geq 1$, with given boundary data in the form of Neumann boundary conditions.
	This work is inspired by the already known representation formula for the $s$-harmonic extension of the fractional Laplacian onto the upper-half space~${\mathbb{R}}^{n+1}_+$ and will lead to the forthcoming Theorem~\ref{Convolution enters the chat}.

	More precisely, we will deal with finite-energy solutions of the following problem
	\begin{equation}\label{Prob with Neumann}
		\begin{cases}
			\div(a(t)\nabla U)=0&\; \text{ in }\; {\mathbb{R}}^{n+1}_+,\\
			-\displaystyle\lim_{t\to 0+} a(t)\partial_t U = f &\; \text{ in }\; {\mathbb{R}}^{n}.
		\end{cases}
	\end{equation}
	
	Here above and in what follows, we assume that~$a: {\mathbb{R}}_+ \rightarrow {\mathbb{R}}_+$ is a continuously differentiable $A_2$-Muckenhoupt\footnote{We recall that a weight~$\omega:{\mathbb{R}}^n\to[0,+\infty)$ belongs to the~$A_2$-Muckenhoupt class if~$\omega$ is locally integrable and there exists~$C>0$ such that, for all balls~$B$,
		$$ \frac1{|B|^2} \left( \int_B \omega(x)\,dx\right)
		\left( \int_B \frac1{\omega(x)}\,dx\right)\le C,$$
		where~$|B|$ denotes the Lebesgue measure of~$B$.
		See e.g.~\cite{MR643158}, where a regularity theory for solutions
		of equations involving Muckenhoupt weights is developed.} weight.
		
	We also recall that we are using the notation~$(x,t)\in\R^n\times(0,+\infty)$ to denote points in~$\R^{n+1}_+$.
	
	\subsection{Functional spaces, notation and results}
	We will now introduce the function spaces used throughout this section.
	
	As customary, $\mathcal{S}({\mathbb{R}}^n)$ will denote the Schwartz space in~${\mathbb{R}}^n$ and~$\cS({\mathbb{R}^n})'$ will denote the space of tempered distributions in~$\mathbb{R}^n$. Also, given~$\psi \in \mathcal{S}({\mathbb{R}}^n)$
	and~$T\in \mathcal{S}'({\mathbb{R}}^n)$, the notation~$
	\langle T, \psi \rangle_{\mathcal{S}({\mathbb{R}}^n)}$
	means that the distribution~$T$ is applied to~$\psi$.
	
	Moreover, we define the space of functions with ``finite energy'' as
	\begin{equation}\label{space X}
		X({\mathbb{R}}^{n+1}_+,a):=\left\{U\in W^{1,1}_{\text{loc}}({\mathbb{R}}^{n+1}_+)\;{\mbox{ s.t. }}\;\int_{{\mathbb{R}}^{n+1}_+}a(t)|\nabla U(x,t)|^2 \,dx\,dt<+\infty\right\}.
	\end{equation}
	
	Given~$\Omega\subseteq{\mathbb{R}}^{n+1}_+$, we define the norm
	$$ \|U\|_{H^1(\Omega,a)}:= \left(\int_{\Omega}a(t)\big(
	|\nabla U(x,t)|^2 + |U(x,t)|^2\big) \,dx\,dt\right)^{1/2} 
	$$
	and the weighted Sobolev space
	\begin{equation*}
		H^1(\Omega,a) :=\Big\{U\in W^{1,1}_{\text{loc}}(\Omega)\;{\mbox{ s.t. }}\; \|U\|_{H^1(\Omega,a)}<+\infty\Big\}.
	\end{equation*}
	Weighted Poincar\'e inequalities (see Section~1 of~\cite{MR643158} for weighted Sobolev
	and Poincar\'e inequalities and Imbedding Theorems) imply that~$X({\mathbb{R}}^{n+1}_+,a) \subseteq H^1_{\text{loc}}({\mathbb{R}}^{n+1}_+, a)$. 
	
	For any~$\lambda\ge0$, we define
	\begin{equation}\label{inf of G }
		m({\lambda}):= \inf_{\underset{\phi(0)=1}
			{\phi \in W_{\text{loc}}^{1,1}(\mathbb{R}_+})}\int_{{\mathbb{R}}_+}a(t)\left(\lambda |\phi(t)|^2+|\phi'(t)|^2\right)\,dt.
	\end{equation}
	This minimization problem will be described in detail in the forthcoming
	Theorem~\ref{Minimization Problem} for~$\lambda >0$.
	
	We also define the space of functions
	\begin{equation*}
		H_a({\mathbb{R}}^n):= \left\{ u \in \mathcal{S}'(\mathbb{R}^n) \;{\mbox{ s.t. }}\;{\mbox{$\hat{u}$ is a measurable function and}}\; [ u]_{H_a(\mathbb{R}^n)}<+\infty \right\},
	\end{equation*}
	where
	$$
	[u]_{H_a(\mathbb{R}^n)}:=\left(\int_{{\mathbb{R}}^n} m(|\xi|^2) |\hat{u}(\xi)|^2 \,d\xi\right)^{1/2}.
	$$
	We point out that the integral above is well-defined, since~$m$ is a measurable function (see Remark~\ref{m is a measurable function}).
	
	Here above and in the rest of the paper, we use the notation~$\mathcal{F}u=\hat u$ to denote the (distributional) Fourier transform
	and~$\mathcal{F}^{-1}u$ to denote the inverse Fourier transform of~$u \in \mathcal{S}'(\mathbb{R}^n)$.
	Given~$V \in \mathcal{S}'(\mathbb{R}^{n+1})$, we denote by~$\mathcal{F}_x V$ the distributional Fourier transform in the~$x$-variable
	by~$\mathcal{F}_{x}^{-1} V$ its inverse, which are given by 
	\begin{eqnarray*} &&
		\langle \mathcal{F}_x V,\psi\rangle_{\mathcal{S}(\mathbb{R}^{n+1})} = \langle V, \mathcal{F}_x \psi\rangle_{\mathcal{S}(\mathbb{R}^{n+1})}
		\\{\mbox{and }}&&
		\langle \mathcal{F}_{x}^{-1} V,\psi\rangle_{\mathcal{S}(\mathbb{R}^{n+1})} = \langle V,  \mathcal{F}_{x}^{-1}  \psi\rangle_{\mathcal{S}(\mathbb{R}^{n+1})},
	\end{eqnarray*}
	for every~$\psi \in \mathcal{S}\left(\mathbb{R}^{n+1}\right)$. Note that Fourier transforming (and inverse Fourier transforming) a tempered distribution in $\mathbb{R}^{N+1}$ in the $x-$variable is a well-defined operation in $\mathcal{S}'(\mathbb{R}^{n+1})$ because the Fourier transform in the $x-$ variable of a function in $\mathcal{S}(\mathbb{R}^{n+1})$ belongs to $\mathcal{S}(\mathbb{R}^{n+1})$, i.e., $\mathcal{F}_x(\mathcal{S}(\mathbb{R}^{n+1}))\subset \mathcal{S}(\mathbb{R}^{n+1})$.
	
	Also, we define the map~$\mathcal{L}_a:H_a(\mathbb{R}^n)\to \mathcal{S}'(\mathbb{R}^n)$ such that, for all~$u\in H_a(\mathbb{R}^n)$,
	\begin{equation}\label{r43308kfjvbnmghjksagfkuegwut}
		\mathcal{L}_a(u) := \mathcal{F}^{-1}(m(|\xi|^2)\hat{u}).\end{equation}
	
	Throughout this section, we say that~$U\in W^{1,1}_{\text{loc}}\left(\mathbb{R}^{n+1}_+\right)$ is a weak solution of~\eqref{Prob with Neumann} if
	\begin{enumerate}
		\item for any~$R>0$, we have that~$a(t)|\nabla U|^2\in L^1(B_{R}^{+})$;
		\item for any $\varphi \in C^{\infty}_c(\mathbb{R}^{n+1})$, it holds that
		\begin{equation}\label{weak solution}
			\int_{\mathbb{R}^{n+1}_+}a(t)\nabla U(x,t) \cdot \nabla \varphi(x,t) \,dx\,dt = \int_{\partial \mathbb{R}^{n+1}_+}f(x) \varphi(x,0)\,dx.
		\end{equation}
	\end{enumerate}
	Moreover, we say that~$U$ is a finite-energy weak solution
	of~\eqref{Prob with Neumann}
	if~$U\in X(\mathbb{R}^{n+1}_+,a)$ and it is a weak solution.
	
	In this setting, we have the following result:
	
	\begin{theorem}\label{Convolution enters the chat}
		Let~$a :\mathbb{R}_+\rightarrow\mathbb{R}_+$ be a continuously differentiable~$A_2$-Muckenhoupt weight.
		Let~$f \in \mathcal{L}_a(H_a(\mathbb{R}^n))\cap L^{1}_{\text{loc}}(\mathbb{R}^n)$.
		
		Then, all the finite-energy weak solutions of~\eqref{Prob with Neumann} are given by 
		\begin{equation}\label{Sol of Prob with Neumann}
			U= \mathcal{F}_x^{-1}\left(\tilde{g}(|\xi|^2,t)\hat{u}\right)+ c,  
		\end{equation}
		where, for all~$\xi\in{\mathbb{R}}^n$, $\tilde{g}(|\xi|^2,\cdot)$ is the even extension to~$t<0$ of the minimizer of~\eqref{inf of G } with~$\lambda := |\xi|^2$,
		$u \in H_a(\mathbb{R}^n)$ is the (unique) solution, in the sense of tempered distributions, of~$\mathcal{L}_a(u) = f$ and~$c\in \mathbb{R}$.
	\end{theorem}
	
	{F}rom Theorem~\ref{Convolution enters the chat}, we obtain the following:
	
	\begin{corollary}\label{Rigidity result}
		Let~$f \in \mathcal{L}_a(H_a(\mathbb{R}^n))\cap L^1_{\text{loc}}(\mathbb{R}^n)$, $U\in X(\mathbb{R}^{n+1},a)$ be a finite-energy weak solution of~\eqref{Prob with Neumann} and~$u\in W^{1,1}_{\text{loc}}(\mathbb{R}^n)\cap H_a(\mathbb{R}^n)$ be the solution of~$\mathcal{L}_a(u)=f$ (in the
		sense of tempered distributions).
		
		If~$u$ is one-dimensional, then, for a.e.~$t\geq 0$, $U(\cdot,t)$ is also one-dimensional.
	\end{corollary}
	
	\begin{remark}\label{BLW}
		In the particular case where~$a(t):= t^{1-2s}$ for some~$s\in(0,1)$, 
		the problem in~\eqref{Prob with Neumann} is related with
		the $s$-harmonic extension of the fractional Laplacian onto the upper half-space~$\mathbb{R}^{n+1}_+$. 
		
		More precisely, when~$a(t)=t^{1-2s}$, by a scaling argument one can check that~~$m(\lambda) = \lambda^{s} m(1)$ (see formula~(4.9) in~\cite{MR3233760}).
		As a result, for every~$\xi\in\R^n\setminus\{0\}$,
		we have that~$m({|\xi|^2}) = |\xi|^{2s} m(1)$, and therefore~$H_a(\mathbb{R}^n)$ boils down to the space of functions~$u\in \mathcal{S}'(\mathbb{R}^n)$
		such that~$\hat{u}$ is a measurable function and 
		$$
		\int_{\mathbb{R}^n}|\xi|^{2s} |\hat{u}(\xi)|^2\,d\xi<+\infty.
		$$ 
		
		Furthermore, as shown in~\cite{MR2354493}, a finite energy solution of 
		\begin{equation*}
			\begin{cases}
				\div(t^{1-2s} \nabla U)=0 &\text{in }\, \mathbb{R}^{n+1}_+,\\
				-\displaystyle\lim_{t\to 0^+}a(t)\partial_t U = f &\text{on }\, \partial\mathbb{R}^{n+1}_+
			\end{cases}
		\end{equation*}
		is given by the convolution of~$u$ with the $s$-Poisson kernel~$P_s$, i.e., $U(x,t) = P_s(\cdot,t)\ast u(x)$. Also, $$-\displaystyle\lim_{t\to 0^+} t^{1-2s} \partial_t U(x,t) = (-\Delta)^s u(x),$$ up to constants, where~$u(x) = U(x,0)$. 
		
		Therefore, if~$u \in H_{t^{1-2s}}(\mathbb{R}^n)$, 
		from Theorem~\ref{Convolution enters the chat} we have that~$P_s (\cdot, t)\ast u(x) = \mathcal{F}_x^{-1}(\tilde{g}(|\xi|^2,t) \hat{u})$, and consequently~$\mathcal{F}_x \left(P_s\right) (\xi, t) = \tilde{g}(|\xi|^2,t)$.  
		
		This also shows that~$\mathcal{F}_x(P_s) (\xi, \cdot)$ is the unique solution of the variational problem~\eqref{inf of G }. 
	\end{remark}
	
	
	\subsection{Towards the proof of Theorem~\ref{Convolution enters the chat}}
	Our strategy to prove Theorem~\ref{Convolution enters the chat} is to first show the uniqueness, up to constants, of finite-energy weak solutions of~\eqref{Prob with Neumann} and secondly to construct an explicit solution.
	
	Showing the uniqueness of finite-energy weak solutions of~\eqref{Prob with Neumann} is rather standard. Indeed, we exploit the methods in Section~4 of~\cite{MR2354493} to extend a weak solution of~\eqref{Prob with Neumann} with~$f\equiv 0$ to the whole space~$\mathbb{R}^{n+1}$, and combine this with a finite-energy Liouville Theorem provided in~\cite{MR4204715} to conclude that a finite-energy weak solution of
	$$
	\div(\tilde{a}(t)\nabla U ) = 0\quad \text{in }\mathbb{R}^{n+1}
	$$
	is constant, where~$\tilde{a}$ is the even extension of~$a$ to the whole of~$\mathbb{R}$.
	
	Constructing a finite-energy weak solution of~\eqref{Prob with Neumann} is instead more involved. Indeed, our method roughly consists of ``Fourier transforming''~\eqref{Prob with Neumann} in the~$x$-variable, thus obtaining a second-order ordinary differential equation, which is satisfied by a minimizer of~\eqref{inf of G }. Then, using the methods in~\cite{MR3233760} one can deduce several properties of the minimizer of~\eqref{inf of G }, which are summarised in Theorem~\ref{Minimization Problem} below.
	
	{F}rom these properties, we show that~\eqref{Sol of Prob with Neumann} is a well-defined tempered distribution in~$\mathbb{R}^{n+1}$ and that it is an element of~$X(\mathbb{R}^{n+1}_+,a).$ 
	{F}urthermore, the limit properties satisfied by the minimizer of~\eqref{inf of G } presented in Theorem~\ref{Minimization Problem} lead to limit properties of~\eqref{Sol of Prob with Neumann}, one being that the Neumann boundary condition in~\eqref{Prob with Neumann} is satisfied. 
	Then, using a density-type argument, we show that~\eqref{Sol of Prob with Neumann} is a weak solution of~\eqref{Prob with Neumann}, thus finishing the proof of Theorem~\ref{Convolution enters the chat}.
	\medskip
	
	We start by dealing with the uniqueness of solutions of~\eqref{Prob with Neumann} in~$X({\mathbb{R}}^{n+1}_+,a)$ up to a constant factor.
	
	\begin{lemma}\label{Uniqueness of finite energy}
		Let~$U_1$, $U_2\in X({\mathbb{R}}^{n+1}_+,a)$ 
		be two finite-energy weak solutions of~\eqref{Prob with Neumann} with boundary datum~$f \in L^{1}_{\text{loc}}(\mathbb{R}^n)$. 
		
		Then, we have that~$U_1 = U_2 + c$, for some~$c\in \mathbb{R}$.
	\end{lemma}
	
	\begin{proof}
		Since~$U_1$ and~$U_2$ are finite-energy weak solutions of~\eqref{Prob with Neumann} we have that~$U:=U_1-U_2$ is a finite-energy
		weak solution of
		\begin{equation}\label{eq in r n+1 +}
			\begin{cases}
				\div(a(t)\nabla U) = 0&\;\text{ in } \; {\mathbb{R}}^{n+1}_+,\\
				\displaystyle\lim_{t\to 0^+}a(t)\partial_t U(x,t) = 0&\;\text{ on }\; \partial {\mathbb{R}}^{n+1}_+.
			\end{cases}
		\end{equation}
		
		Now, for any~$t\in \mathbb{R}\setminus \{0\}$, we define
		$$ \widetilde{a}(t):=\begin{cases}
			a(t) &{\mbox{ if }} t>0,\\
			a(-t) &{\mbox{ if }} t<0
		\end{cases} 
		\qquad {\mbox{and}} \qquad
		\tilde{U}(x,t) :=\begin{cases}
			U(x,t) &{\mbox{ if }} t\ge 0,\\
			U(x,-t) &{\mbox{ if }} t<0.
		\end{cases} $$    
		Observe that, by construction, $\tilde{U}$ has finite energy in~$\mathbb{R}^{n+1}$,
		namely~$\tilde{U}\in X({\mathbb{R}}^{n+1},\widetilde{a})$.
		Hence (by the weighted Poincar\'e inequalities in Section~1 of~\cite{MR643158}),
		we have that~$
		\tilde{U}\in H^1_{\text{loc}}(\mathbb{R}^{n+1},\tilde{a})$.
		
		Furthermore, we claim that~$\tilde{U}$ is a finite-energy weak solution of 
		\begin{equation}\label{equation r n+1}
			\div(\tilde{a}(t) \nabla \tilde{U}) = 0\quad\text{in }\mathbb{R}^{n+1}.
		\end{equation}
		To establish the validity of this claim, we will show that,
		for any~$\varphi \in C^{\infty}_c(\mathbb{R}^{n+1})$,
		$$
		\int_{\mathbb{R}^{n+1}}\tilde{a}(t) \nabla \tilde{U}\cdot \nabla \varphi \,dx\,dt = 0.
		$$
		Indeed, for	~$\varphi \in C^{\infty}_c(\mathbb{R}^{n+1})$,
		using twice the fact that~$U$ is a weak solution of~\eqref{eq in r n+1 +}, we find that
		\begin{equation*}
			\begin{split}
				\int_{\mathbb{R}^{n+1}}\tilde{a}(t) \nabla \tilde{U}\cdot \nabla \varphi \,dx\,dt &= \int_{\mathbb{R}^{n+1}_+}a(t) \nabla U\cdot \nabla \varphi \,dx\,dt + \int_{\mathbb{R}^{n+1}_{-}}a(-t) \nabla \tilde{U}\cdot \nabla \varphi \,dx\,dt \\
				&= \int_{\mathbb{R}^{n+1}_{-}}a(-t) \big(\nabla_x U(x,-t), -\partial_t U(x,-t)\big)\cdot \nabla \varphi (x,t)\,dx\,dt\\
				&= \int_{\mathbb{R}^{n+1}_{+}}a(\tau) \nabla U(x,\tau) \cdot \big(
				\nabla_x \varphi(x,-\tau), -\partial_\tau\varphi(x,-\tau) \big) \,dx\,d\tau \\
				&= \int_{\mathbb{R}^{n+1}_{+}}a(\tau) \nabla U(x,\tau) \cdot \nabla\tilde\varphi(x,t)\,dx\,d\tau\\
				&=0,
			\end{split}
		\end{equation*}
		where~$\tilde\varphi(x,t):=\varphi(x,-t)$ for all~$(x,t)\in\mathbb{R}^{n+1}_{+}$.
		Therefore~\eqref{equation r n+1} is established.
		
		{F}rom~\eqref{equation r n+1} and the fact that the energy functional
		$$ \int_{\mathbb{R}^{n+1}} \tilde{a}(t)|\nabla \tilde{U}|^2\,dx\,dt $$
		is strictly convex, we deduce that~$\tilde{U}$ is the (unique)
		minimizer up to an additive constant.
		
		Hence, as a direct consequence of the Finite-Energy Liouville Theorem in~\cite{MR4204715} we have that~$\tilde{U}$ is constant,
		and thus so is~$U$, as desired.
	\end{proof}
	
	Before starting the construction of solutions of~\eqref{Prob with Neumann} we give two properties of Muckenhoupt weights that will be useful later on to show that~\eqref{Sol of Prob with Neumann} is a well-defined tempered distribution:
	one property refers to the integrability of Muckenhoupt weights and the other is a growth upper bound of $A_2$-Muckenhoupt weights. We believe
	that properties like these should be known in the literature, but we did not find any references that include them, as such, for the readers' convenience, we state these results here and include detailed proofs in Appendix~\ref{gyu4io3vbcnx76859430}.
	
	\begin{lemma}\label{Muckenhoupt is not L1}
		Let~$\omega:  {\mathbb{R}}_+ \rightarrow {\mathbb{R}}_+$ be an~$A_2$-Muckenhoupt weight. Then, $\omega \notin L^1(\mathbb{R}_+)$. 
	\end{lemma}
	
	\begin{lemma}\label{Growth bound on Muckenhoupt}
		Let~$\omega: {\mathbb{R}}_+ \rightarrow {\mathbb{R}}_+$ be an~$A_2$-Muckenhoupt weight.
		
		Then, there exists~$C>0$ such that, for all~$t\ge0$,
		\begin{equation*}
			\int_{0}^{t}\frac{d\tau}{\omega(\tau)}\leq C(1+t^2).
		\end{equation*}
	\end{lemma}
	
	We now start the construction of solutions of~\eqref{Prob with Neumann} by presenting some properties of the (unique) solution of the one-dimensional minimization problem described in~\eqref{inf of G }.
	To this end, we introduce the space
	$$ H^1_1(\mathbb{R}_+,a):= \left\{\phi \in W^{1,1}_{\text{loc}}(\mathbb{R}_+)\;\text{s.t.}\; \displaystyle\int_{\mathbb{R}_+}a(t)\left(|\phi(t)|^2+|\phi'(t)|^2\right)\,dt<+\infty \,\text{and}\,\, \phi(0)=1
	\right\}.$$
	Also, for~$\lambda>0$ and~$\phi \in H^1_1(\mathbb{R}_+,a)$ we define the functional 
	\begin{equation*}
		G_{\lambda}(\phi):= \int_{\mathbb{R}_+}a(t)\left( \lambda|\phi(t)|^2+ |\phi'(t)|^2\right)\,dt.
	\end{equation*}
	
	The problem of minimizing~$G_{\lambda}$ over~$H^{1}_{1}(\mathbb{R}_+,a)$ has already been studied in~\cite{MR3233760}
	in the particular case in which~$a(t)= t^{1-2s}$ with~$s\in (0,1)$
	(see Theorem~4.2 in~\cite{MR3233760}). In fact, following the same arguments one arrives at the same conclusions in our setting:
	
	\begin{theorem}\label{Minimization Problem}
		For all~$\lambda>0$, there exists a unique~$g(\lambda,\cdot)\in H^{1}_{1}(\mathbb{R}_+, a)$ such that 
		\begin{equation*}
			m(\lambda):=\inf_{u \in H_{1}^{1}(\mathbb{R}_+,a)} G_{\lambda}(u) = G_{\lambda}(g(\lambda,\cdot)).
		\end{equation*}
		Moreover, $g(\lambda,\cdot)\in C^{2}(\mathbb{R}_+)$ and\footnote{We point out that
			in Theorem~4.2 in~\cite{MR3233760}, one obtains that~$g(\lambda,\cdot)$ is a smooth function because in that theorem the weight~$a$ is smooth in~$\mathbb{R}_+$. In our general setting, we assume that~$a$ is only continuously differentiable in~$\mathbb{R}_+$, and therefore~$g(\lambda,\cdot)$ is only twice continuously differentiable.} satisfies the following properties:
		\begin{enumerate}
			\item[(1)] $g(\lambda,\cdot)$ is a solution of
			\begin{equation}\label{Fourier transform of eq for kernel}
				\begin{cases}
					a(t)\partial^2_{t} g(\lambda,t) + a'(t)\partial_tg(\lambda,t)- \lambda a(t)g(\lambda,t)=0 \; \text{ in }\; \mathbb{R}_+,\\
					g(\lambda,0)=1;
				\end{cases}
			\end{equation}
			\item[(2)] $\partial_tg(\lambda,t) \leq 0$ for all~$t\in\R_+$;
			\item[(3)] $\displaystyle\lim_{t \to +\infty} g(\lambda,t) = 0$;
			\item[(4)] $\displaystyle-\lim_{t \to 0^+}a(t)\partial_t g(\lambda,t) = m(\lambda)$.
		\end{enumerate}
	\end{theorem}
	
	\begin{remark}\label{m is a measurable function}
		Firstly note that, by definition, $m(\lambda)$ is an increasing function of~$\lambda$ and that~$m(\lambda)\geq 0$ for any~$\lambda >0$. Furthermore, for~$\lambda$, $\eps>0$ we have that
		\begin{equation*}
			\begin{split}&
				m(\lambda+\eps) =
				\inf_{u \in H_{1}^{1}(\mathbb{R}_+,a)} G_{\lambda+\eps}(u) 
				= \inf_{u \in H_{1}^{1}(\mathbb{R}_+,a)}\left(G_{\lambda}(u)+\eps \int_{\mathbb{R}_+}a(t)|u(t)|^2\,dt\right)\\
				&\qquad\leq m(\lambda)+\eps\int_{\mathbb{R}_+}a(t)|g(\lambda,t)|^2\,dt
				\leq m(\lambda) + \frac{m(\lambda)}{\lambda}\eps.
			\end{split}
		\end{equation*}
		This, together with the fact that~$m(\lambda+\eps)\geq m(\lambda)$, implies that 
		$$
		|m(\lambda+\eps)-m(\lambda)|\leq  \frac{m(\lambda)}{\lambda}\eps,
		$$
		and therefore~$m(\lambda)$ is locally Lipschitz in~$\mathbb{R}_+$.
		
		Also note that, in the case~$\lambda = 0$ in~\eqref{inf of G }, by inputing~$\phi(t)\equiv 1$, we find~$m(0) = 0$. In particular, $m(\lambda)$ is a measurable function in~$[0,+\infty)$.
	\end{remark}
	
	\begin{remark}\label{bound on derivative of g}
	A consequence of Theorem~\ref{Minimization Problem} is that,
		for all~$t\in \mathbb{R}_+$,
		\begin{equation}\label{fyuiwoqvbcnx65748398765vbn}
			\left|a(t) \partial_tg(\lambda,t)\right|\leq m(\lambda).\end{equation}
		Indeed, 
		from the properties of~$g(\lambda,\cdot)$ in Theorem~\ref{Minimization Problem}, we see that~$g(\lambda,t)\in[0,1]$ for all~$t\in\R_+$.
		Thus,
		using the equation in~\eqref{Fourier transform of eq for kernel},
		$$
		-\frac{\partial}{\partial t}\big(a(t)\partial_t g(\lambda,t)\big)
		=-\big(a'(t)\partial_tg(\lambda,t)+a(t)\partial^2_tg(\lambda,t)\big)
		= -\lambda a(t)g(\lambda,t)\le0.
		$$
		Therefore, 
		$$
		\left|a(t) \partial_tg(\lambda,t)\right|= -a(t)\partial_tg(\lambda,t) \leq 
		-\displaystyle\lim_{t\to 0^+}a(t)\partial_tg(\lambda,t)= m(\lambda),
		$$ which establishes~\eqref{fyuiwoqvbcnx65748398765vbn}.
	\end{remark}
	
	We now present a technical, but useful result, which is a density argument of classical flavor, and whose detailed proof is contained in Appendix~\ref{gyu4io3vbcnx76859430}.
	
	\begin{lemma}\label{Density of Schwartz} We have that~$\mathcal{S}(\mathbb{R}^n)\subset H_a(\mathbb{R}^n)$ and,
		for any~$v \in H_a(\mathbb{R}^n)$, there exists a sequence~$v_j \in \mathcal{S}(\mathbb{R}^n)$ such that 
		\begin{equation}\label{4358rufinittesalimit}
			\lim_{j\to+\infty}[v-v_j]_{H_a(\R^n)}= 0.
		\end{equation}
	\end{lemma}
	
	
	With the work done so far, we now show that the expression in~\eqref{Sol of Prob with Neumann} is well-defined as an element of~$\mathcal{S}'(\mathbb{R}^{n+1})$, as stated in the following result:
	
	\begin{lemma}\label{Defining the Distribution}
		For every~$v \in H_a(\mathbb{R}^n)$, we have that
		$$\mathcal{F}_x^{-1}\left(\tilde{g}(|\xi|^2,t)\hat{v}\right)\in
		\mathcal{S}'(\mathbb{R}^{n+1}),$$
		where~$\tilde{g}(|\xi|^2,\cdot)$ is the even extension to~$t<0$ of the function~$g(|\xi|^2,\cdot)$ given by Theorem~\ref{Minimization Problem}. 
	\end{lemma}
	
	\begin{proof}
		Let~$v\in H_a(\mathbb{R}^n)$.
		We define,
		for all~$\varphi\in{\mathcal{S}}(\mathbb{R}^{n+1})$,
		\begin{eqnarray*} \langle \mathcal{F}_x^{-1}\left(\tilde{g}(|\xi|^2,t)\hat{v}\right)\,,\varphi\rangle_{\mathcal{S}(\mathbb{R}^{n+1})}&:=&
			\langle \tilde{g}(|\xi|^2,t)\hat{v}\,,\mathcal{F}_x^{-1}\varphi\rangle_{\mathcal{S}(\mathbb{R}^{n+1})}\\
			&:=&\int_{\mathbb{R}^{n+1}}\tilde{g}(|\xi|^2,t)\,\hat{v}(\xi)\,
			\mathcal{F}_x^{-1}\varphi(\xi,t)\,d\xi\,dt.
		\end{eqnarray*}
		With this setting, in order
		to establish Lemma~\ref{Defining the Distribution}, we need to check that
		if~$\psi_k$ converges to~$\psi$ in~${\mathcal{S}}(\mathbb{R}^{n+1})$ then
		\begin{equation}\label{ftydusivfebdwjsq4532617845tyergfhkjskh00}
			\lim_{k\to+\infty}
			\langle \tilde{g}(|\xi|^2,t)\hat{v}\,,\psi_k-\psi\rangle_{\mathcal{S}(\mathbb{R}^{n+1})}
			=0.\end{equation}
		To this end, we observe that
		\begin{equation}\label{ftydusivfebdwjsq4532617845tyergfhkjskh0}\begin{split}
				&\langle \tilde{g}(|\xi|^2,t)\hat{v}\,,\psi_k-\psi\rangle_{\mathcal{S}(\mathbb{R}^{n+1})}
				\\=\;&
				\int_{\mathbb{R}^{n+1}}\tilde{g}(|\xi|^2,t)\,\hat{v}(\xi)\big(\psi_k(\xi,t)-\psi(\xi,t)\big)\,d\xi\,dt\\=\;&
				\int_{\mathbb{R}^{n+1}}
				\big(\tilde{g}(|\xi|^2,t)-1\big)\hat{v}(\xi)\big(\psi_k(\xi,t)-\psi(\xi,t)\big)\,d\xi\,dt
				\\&\qquad +\int_{\mathbb{R}^{n+1}}\hat{v}(\xi)\big(\psi_k(\xi,t)-\psi(\xi,t)\big)\,d\xi\,dt.
			\end{split}
		\end{equation}
		
		We claim that
		\begin{equation}\label{ftydusivfebdwjsq4532617845tyergfhkjskh}
			\lim_{k\to+\infty}
			\int_{\mathbb{R}^{n+1}}
			\big(\tilde{g}(|\xi|^2,t)-1\big)\hat{v}(\xi)\big(\psi_k(\xi,t)-\psi(\xi,t)\big)\,d\xi\,dt=0.
		\end{equation}
		To prove this claim, we observe that,
		thanks to formula~\eqref{fyuiwoqvbcnx65748398765vbn}
		in Remark~\ref{bound on derivative of g},
		for all~$t\in\R$,
		\begin{equation}\label{servdopo0854oi0987}
			\left|\tilde{g}(|\xi|^2,t)-1\right|\le \int_0^{|t|}|\partial_\sigma\tilde{g}(|\xi|^2,\sigma)|\,d\sigma\le m(|\xi|^2)\int_0^{|t|}\frac{d\sigma}{a(\sigma)}.
		\end{equation}
		Thus, Lemma~\ref{Growth bound on Muckenhoupt} gives that
		\begin{eqnarray*}
			\left|\tilde{g}(|\xi|^2,t)-1\right|\le C\,m(|\xi|^2)(1+|t|^2).
		\end{eqnarray*}
		As a consequence of this and the H\"older inequality,
		\begin{equation*}
			\begin{split}
				\left|\int_{\mathbb{R}^{n+1}} \right. & \left.
				\big(\tilde{g}  (|\xi|^2,t)-1\big)\hat{v}(\xi)\big(\psi_k(\xi,t)-\psi(\xi,t)\big)\,d\xi\,dt\right|
				\\
				\leq\;&
				\int_{\mathbb{R}^{n+1}}\left|\tilde{g}(|\xi|^2,t)-1\right||\hat{v}(\xi)|\,
				\big|\psi_k(\xi,t)-\psi(\xi,t)\big|\,d\xi\,dt\\
				\leq\;& C \int_{\mathbb{R}^{n+1}}m(|\xi|^2)(1+|t|^2)|\hat{v}(\xi)|\,
				\big|\psi_k(\xi,t)-\psi(\xi,t)\big|\,d\xi\,dt \\
				=\;&C\int_{\mathbb{R}}\frac1{(1+|t|^2)}\left(
				\int_{\mathbb{R}^{n}}m(|\xi|^2) (1+|t|^2)^2|\hat{v}(\xi)|\,\big|\psi_k(\xi,t)-\psi(\xi,t)\big|\,d\xi\right)\,dt\\
				\leq\;& C\int_{\mathbb{R}}\Bigg[\frac1{(1+|t|^2)}\left(
				\int_{\mathbb{R}^{n}} m(|\xi|^2)|\hat{v}(\xi)|^2\,d\xi \right)^{1/2}
				\\&\qquad\times
				\left(\int_{\mathbb{R}^{n}}
				m(|\xi|^2)(1+|t|^2)^4\big|\psi_k(\xi,t)-\psi(\xi,t)\big|^2\,d\xi \right)^{1/2}\Bigg]\,dt
				\\
				=\;&C [v]_{H_a(\mathbb{R}^n)}\int_{\mathbb{R}}\frac{1}{(1+|t|^2)}\left(\int_{\mathbb{R}^{n}}m(|\xi|^2)(1+|t|^{2})^4\big|\psi_k(\xi,t)-\psi(\xi,t)\big|^2\,d\xi\right)^{1/2}\,dt.
			\end{split}
		\end{equation*}
		
		Notice that~$m(|\xi|^2)\le m(1)(1+|\xi|^2)$, and therefore
		\begin{equation*}
			\begin{split}&
				\left|\int_{\mathbb{R}^{n+1}}
				\big(\tilde{g}(|\xi|^2,t)-1\big)\hat{v}(\xi)\big(\psi_k(\xi,t)-\psi(\xi,t)\big)\,d\xi\,dt\right|\\
				\leq\;& C m(1) [v]_{H_a(\mathbb{R}^n)}\int_{\mathbb{R}}\frac{1}{1+|t|^2}\left(\int_{\mathbb{R}^{n}}(1+|\xi|^2)(1+|t|^{2})^4\big|\psi_k(\xi,t)-\psi(\xi,t)\big|^2\,d\xi\right)^{1/2}\,dt.
			\end{split}
		\end{equation*}
		Since~$\psi_k$ converges to~$\psi$ in~$\mathcal{S}(\mathbb{R}^{n+1})$, this formula
		gives the desired limit in~\eqref{ftydusivfebdwjsq4532617845tyergfhkjskh}.
		
		We now show that
		\begin{equation}
			\label{ftydusivfebdwjsq4532617845tyergfhkjskh2}
			\lim_{k\to+\infty}\int_{\mathbb{R}^{n+1}}\hat{v}(\xi)\big(\psi_k(\xi,t)-\psi(\xi,t)\big)\,d\xi\,dt=0.
		\end{equation}
		To this end, note that since~$\psi$, $\psi_k\in\mathcal{S}(\mathbb{R}^{n+1})$, we have that 
\begin{eqnarray*}
		\Psi(x) &:=&\int_{\mathbb{R}}\psi(x,t)\,dt\in \mathcal{S}(\mathbb{R}^{n}) 
		\\{\mbox{and }}\quad 
		\Psi_k(x) &:=&\int_{\mathbb{R}}\psi_k(x,t)\,dt\in \mathcal{S}(\mathbb{R}^{n}) .
		\end{eqnarray*}
		Furthermore, the convergence of~$\psi_k$ to~$\psi$ in~$\mathcal{S}(\mathbb{R}^{n+1})$ implies the convergence of~$\Psi_k$ to~$\Psi$ in~$\mathcal{S}(\mathbb{R}^{n})$. 
		Hence, the fact that $\hat{v}\in \mathcal{S}'(\mathbb{R}^n)$ leads to
		\begin{eqnarray*}
		&&
			\lim_{k\to+\infty}\int_{\mathbb{R}^{n+1}}\hat{v}(\xi)\big(\psi_k(\xi,t)-\psi(\xi,t)\big)\,d\xi\,dt =
			\lim_{k\to+\infty}\int_{\mathbb{R}^{n}}\hat{v}(\xi)\left(\int_{\R}\big(\psi_k(\xi,t)-\psi(\xi,t)\big)\,dt\right)\,d\xi \\&&\qquad\qquad=
			\lim_{k\to+\infty}\int_{\mathbb{R}^{n}}\hat{v}(\xi)\big(\Psi_k(\xi)-\Psi(\xi)\big)\,d\xi = 0,
		\end{eqnarray*}
		which establishes~\eqref{ftydusivfebdwjsq4532617845tyergfhkjskh2}
		
{F}rom~\eqref{ftydusivfebdwjsq4532617845tyergfhkjskh0},
		\eqref{ftydusivfebdwjsq4532617845tyergfhkjskh}
		and~\eqref{ftydusivfebdwjsq4532617845tyergfhkjskh2},
		we obtain the desired result in~\eqref{ftydusivfebdwjsq4532617845tyergfhkjskh00}.
	\end{proof}
	
	Having shown that~\eqref{Sol of Prob with Neumann} is a well-defined tempered distribution, we now prove that it also is an element of~$X(\mathbb{R}^{n+1}_+, a)$:
	
	\begin{lemma}\label{Gradient of U}
		Let~$v\in H_a(\mathbb{R}^n)$ and~$V:= \mathcal{F}_{x}^{-1}\left(\tilde{g}(|\xi|^2,t)\hat{v}\right)\in \mathcal{S}'(\mathbb{R}^{n+1})$.
		
		Then,
		\begin{enumerate}[label=(\subscript{P}{{\arabic*}})]
			\item The distributional gradient of~$V$ is
			\begin{equation*}
				\nabla V = (\mathcal{F}_x^{-1}(i\xi \tilde{g}(|\xi|^2,t)\hat{v}), \mathcal{F}_x^{-1}(\partial_t \tilde{g}(|\xi|^2,t)\hat{v}));
			\end{equation*}
			\item $\nabla V \in L^2_{\text{loc}}({\mathbb{R}}^{n+1}_+)$;
			\item $V$ can be identified with a measurable function and~$V\in L^{2}_{\text{loc}}({\mathbb{R}}^{n+1}_+)$;
			\item $V$ has finite energy, i.e.,
			\begin{equation*}
				\int_{{\mathbb{R}}^{n+1}_+}a(t)|\nabla V|^2 \,dx\, dt<+\infty,
			\end{equation*}
			and consequently~$V\in X(\mathbb{R}^{n+1}_+,a)$.
		\end{enumerate}
	\end{lemma}
	
	\begin{proof}
		We start by showing~$(P_1)$. Let~$\psi \in\mathcal{S}({\mathbb{R}}^{n+1})$
		and notice that, for all~$i \in \{1,\dots,n\}$,
		\begin{equation*}
			\begin{split}
				&   \left\langle \partial_{x_i}V, \psi \right\rangle_{\mathcal{S}(\mathbb{R}^{n+1})} 
				= -\left\langle V, \partial_{x_i}\psi \right\rangle_{\mathcal{S}(\mathbb{R}^{n+1})} = \left\langle \tilde{g}(|\xi|^2,t)\hat{v}, i\xi_i\mathcal{F}_{\xi}^{-1}\psi \right\rangle_{\mathcal{S}(\mathbb{R}^{n+1})}\\
				&\qquad\qquad= \left\langle \mathcal{F}_x^{-1}\left(i\xi_i\tilde{g}(|\xi|^2,t)\hat{v}\right), \psi\right\rangle_{\mathcal{S}(\mathbb{R}^{n+1})}.
			\end{split}
		\end{equation*}
		Moreover, by definition,
		\begin{eqnarray*}
			&&       \left\langle \partial_{t}V, \psi \right\rangle_{\mathcal{S}(\mathbb{R}^{n+1})}= -\left\langle V, \partial_{t}\psi \right\rangle_{\mathcal{S}(\mathbb{R}^{n+1})}
			= -\left\langle \tilde{g}(|\xi|^2,t)\hat{v}, \partial_{t}\mathcal{F}_{\xi}^{-1}\psi \right\rangle_{\mathcal{S}(\mathbb{R}^{n+1})} \\
			\\&&\qquad \qquad= -\int_{{\mathbb{R}}^n}\left(
			\hat{v}(\xi) \int_{\mathbb{R}} \tilde{g}(|\xi|^2,t)\partial_t \mathcal{F}_{\xi}^{-1} \psi(\xi,t) \,dt\right)\,d\xi.\end{eqnarray*}
		Thus, integrating by parts in~$t$
		(and recalling the decay properties of~$\tilde{g}$ at infinity, as given in~$(3)$ of Theorem~\ref{Minimization Problem}), we find that
		$$  \left\langle \partial_{t}V, \psi \right\rangle_{\mathcal{S}(\mathbb{R}^{n+1})}
		=\int_{{\mathbb{R}}^{n+1}}\hat{v}(\xi)\partial_t \tilde{g}(|\xi|^2,t) \mathcal{F}_{\xi}^{-1}\psi (\xi,t)\,d\xi\, dt
		= \left\langle \mathcal{F}_{x}^{-1}(\partial_t \tilde{g}(|\xi|^2,t) \hat{v}), \psi \right\rangle_{\mathcal{S}(\mathbb{R}^{n+1})}.
		$$ 
		The proof of~$(P_1)$ is thereby complete. 
		
		We now check~$(P_2)$. For this, we observe that, applying the Plancherel identity for a given~$t\in\mathbb{R}_+$,
		\begin{equation*}
			\begin{split}
				&\int_{{\mathbb{R}}^{n+1}_+}a(t) |\nabla V(x,t)|^2\, dx\, dt
				= \int_{\mathbb{R}_+}\left(a(t) \int_{{\mathbb{R}}^n} \big(
				|\xi|^2 |g(|\xi|^2,t)|^2+ |\partial_t g(|\xi|^2,t)|^2\big) |\hat{v}(\xi)|^2\,d\xi \right)\,dt \\
				&\qquad\qquad= \int_{{\mathbb{R}}^n}m(|\xi|^2)|\hat{v}(\xi)|^2 \,d\xi<+\infty. 
			\end{split}
		\end{equation*}
		Since~$a$ is continuous and strictly positive, we infer from this computation that~$\nabla V\in L^{2}_{\text{loc}}({\mathbb{R}}^{n+1}_+)$, which establishes~$(P_2)$.
		
		Then, by Corollary~2.1 in~\cite{MR0074787} we conclude that~$V$ can be identified with a measurable function~$V \in L^2_{\text{loc}}({\mathbb{R}}^{n+1}_+)$,
		which gives~$(P_3)$. Therefore,
		we conclude that~$V \in X(\mathbb{R}^{n+1}_+,a)$, which completes the proof of~$(P_4)$.
	\end{proof}
	
	To prove Theorem~\ref{Convolution enters the chat} we also
	need to show that~\eqref{Sol of Prob with Neumann} satisfies the boundary condition
	in~\eqref{Prob with Neumann}. To this end, we study the asymptotic behaviour of~$V:= \mathcal{F}_x^{-1}(\tilde{g}(|\xi|^2,t)\hat{v})$, when~$v \in H_a({\mathbb{R}}^n)$. The result needed for our purposes goes as follows:
	
	\begin{lemma}\label{Asymptotic Properties}
		Let~$v\in H_a({\mathbb{R}}^n)$ and~$V(x,t) := \mathcal{F}_x^{-1}(\tilde{g}(|\xi|^2,t)\hat{v})\in \mathcal{S}'(\mathbb{R}^{n+1})$. 
		
		Then,
		\begin{enumerate}[label=(\subscript{P}{{\arabic*}})]\addtocounter{enumi}{4}
			\item $\displaystyle\lim_{t\to 0^+}\|V(\cdot,t) - v(\cdot)\|_{L^2(\mathbb{R}^n)} = 0$;
			\item $\displaystyle\lim_{t\to 0^+} a(t) \partial_t V = -\mathcal{F}_x^{-1}(m(|\xi|^2)\hat{v}(\xi))$ in~$\mathcal{S}'(\mathbb{R}^{n+1})$.
		\end{enumerate}
	\end{lemma}
	
	\begin{proof}   
		We have that
		\begin{equation}\label{1234567qwertyui9876}
			\hat{v} \in L^2(\mathbb{R}^n\setminus B_1).\end{equation}
		Indeed, the monotonicity of~$m$ gives that
		\begin{equation*}
			\int_{\mathbb{R}^n\setminus B_1}|\hat{v}(\xi)|^2\,d\xi
			\leq m(1)^{-1}\int_{\mathbb{R}^n\setminus B_1}m(|\xi|^2)|\hat{v}(\xi)|^2\, d\xi \leq m(1)^{-1}[v]_{H_a(\mathbb{R}^n)}^2<+\infty,
		\end{equation*} which establishes~\eqref{1234567qwertyui9876}.
		
		Moreover, thanks to the properties of~$g$ in Theorem~\ref{Minimization Problem}, we see that~$g(|\xi|^2,t)\in[0,1]$
		and~$g(|\xi|^2,0)=1$. Thanks to these facts and~\eqref{1234567qwertyui9876}, we
		can employ the Dominated Convergence Theorem to conclude that
		\begin{equation}\label{t43ohjhfgiuwehfdt748itykfjdsh87654}
			\lim_{t\to0^+}\int_{\mathbb{R}^n\setminus B_1}\left(g(|\xi|^2,t)-1\right)^2|\hat{v}(\xi)|^2\,d\xi = 0. 
		\end{equation}
		
		Furthermore, the estimate in~\eqref{servdopo0854oi0987} and the monotonicity of~$m$ give that
		\begin{equation}\label{843920oedjshfhja54yzdhgli0}
			\begin{split}
				\int_{B_1}\left(g(|\xi|^2,t)-1\right)^2|\hat{v}(\xi)|^2\,d\xi 
				&\leq \left(\int_0^t \frac{1}{a(\sigma)} \,d\sigma\right)^2
				\int_{B_1} (m(|\xi|^2))^2 |\hat{v}(\xi)|^2\,d\xi \\
				&\leq m(1) [v]_{H_a(\mathbb{R}^n)}^2 \left(\int_0^t \frac{1}{a(\sigma)} \,d\sigma\right)^2.
			\end{split}
		\end{equation}
		We point out that
		\begin{equation}\label{843920oedjshfhja54yzdhgli}
			\lim_{t\to0^+} \int_0^t \frac{1}{a(\sigma)} \,d\sigma=0.
		\end{equation}
		Indeed, we apply the definition of~$A_2$-Muckenhoupt
		weight to find that, for every interval~$I$,
		$$ \left(\int_I a(\sigma)\,d\sigma\right) \left(\int_I \frac1{a(\sigma)}\,d\sigma\right)
		\le C|I|^2,$$
		and therefore, from the fact that~$a$ is locally integrable, we deduce that~$1/a$ is also
		locally integrable. 
		
		Accordingly, the weight~$1/a$ is also in the~$A_2$-Muckenhoupt class.
		As a consequence, we can employ 
		Proposition~7.2.8 in~\cite{MR3243734} (with~$p:=2$)
		and obtain that there exist~$C>0$, depending only on~$a$,
		and~$\delta \in (0,1)$ such that, for every interval~$I \subset \mathbb{R}_+$ and every measurable subset~$J$ of~$I$,
		\begin{equation*}
			\frac{\displaystyle\int_J \frac1{a(\sigma)}\,d\sigma}{\displaystyle\int_I \frac1{a(\sigma)}\,d\sigma}
			\leq C \left(\frac{|J|}{|I|}\right)^{\delta}
			.
		\end{equation*}
		Using this inequality with~$I:=(0,10)$ and~$J:=(0,t)$ with~$t\in(0,10)$, we thus find that
		$$   \frac{\displaystyle\int_0^t \frac1{a(\sigma)}\,d\sigma}{\displaystyle \int_0^{10} \frac1{a(\sigma)}\,d\sigma}\leq C \left(\frac{t}{10}\right)^{\delta},
		$$ which entails~\eqref{843920oedjshfhja54yzdhgli}.
		
		We thus use~\eqref{843920oedjshfhja54yzdhgli} into~\eqref{843920oedjshfhja54yzdhgli0}
		and obtain that
		$$\lim_{t\to0^+} \int_{B_1}\left(g(|\xi|^2,t)-1\right)^2|\hat{v}(\xi)|^2\,d\xi 
		=0.$$
		{F}rom this and~\eqref{t43ohjhfgiuwehfdt748itykfjdsh87654}
		we conclude that 
		\begin{equation*}\lim_{t\to0^+}
			\int_{\mathbb{R}^n}\left(g(|\xi|^2,t)-1\right)^2 |\hat{v}(\xi)|^2 \,d\xi = 0.
		\end{equation*}
		Consequently, the Plancherel identity gives that
		\begin{equation*}
			\begin{split}&\lim_{t\to0^+}
				\int_{\mathbb{R}^n}\left|V(x,t)-v(x)\right|^2 \,dx =\lim_{t\to0^+} \int_{\mathbb{R}^n}\left|\mathcal{F}_x(V)(\xi,t)-\hat{v}(\xi)\right|^2 \,d\xi \\
				&\qquad\qquad = \lim_{t\to0^+}\int_{\mathbb{R}^n}\left(g(|\xi|^2,t)-1\right)^2|\hat{v}(\xi)|^2 \,d\xi = 0.
			\end{split}
		\end{equation*}
		This proves~$(P_5)$ and we now focus on the proof of~$(P_6)$.
		
		We observe that, for any~$\psi\in \mathcal{S}({\mathbb{R}}^n)$ and~$t>0$,
		\begin{equation*}
			\begin{split}
				&    \left|\left\langle a(t)\mathcal{F}_x^{-1}(\partial_t g(|\xi|^2,t)\hat{v})-\mathcal{F}_x^{-1}(m(|\xi|^2)\hat{v}),\psi \right\rangle_{\mathcal{S}(\mathbb{R}^n)}\right|\\=\;& \left|\left\langle (-a(t)\partial_t g(|\xi|^2,t)-m(|\xi|^2))\hat{v},\mathcal{F}_x^{-1}{\psi} \right\rangle_{\mathcal{S}(\mathbb{R}^n)}\right|\\
				=\;& \left|\int_{{\mathbb{R}}^n} \big(-a(t)\partial_t g(|\xi|^2,t)-m(|\xi|^2)\big|\hat{v}(\xi)\mathcal{F}_x^{-1}{\psi}(\xi)\, d\xi\right|.
			\end{split}
		\end{equation*}
		We recall that~$|a(t)\partial_t g(|\xi|^2,t)|\le m(|\xi|^2)$, thanks to formula~\eqref{fyuiwoqvbcnx65748398765vbn}.
		Moreover, we have that~$v$, $\psi \in H_a({\mathbb{R}}^n)$, in light of Lemma~\ref{Density of Schwartz}. Also, by~$(4)$ of Theorem~\ref{Minimization Problem} we know that
		$$\lim_{t\to 0^+}a(t)\partial_t g(|\xi|^2,t) = -m(|\xi|^2).$$
		Therefore, using the Dominated Convergence Theorem, we conclude that 
		\begin{equation*}
			\lim_{t \to 0^+}\left\langle a(t)\mathcal{F}_x^{-1}(\partial_t g(|\xi|^2,t)\hat{v})-\mathcal{F}_x^{-1}(m(|\xi|^2)\hat{v}),\psi \right\rangle_{\mathcal{S}(\mathbb{R}^n)}=0,
		\end{equation*}
		which gives the desired convergence in~$(P_6)$.
	\end{proof}
	
	Now we present a lemma that extends the density-type result in
	Lemma~\ref{Density of Schwartz} to expressions of the form~$\mathcal{F}_{x}^{-1}(\tilde{g}(|\xi|^2,t)\hat{v})$, with~$v \in H_a(\mathbb{R}^n)$.
	
	The usefulness of the following lemma lies in the fact that it will allow us to show that~\eqref{Sol of Prob with Neumann} is a weak solution of~\eqref{Prob with Neumann}. 
	
	\begin{lemma}\label{Convergence of the gradients}
		Let~$v \in H_a({\mathbb{R}}^n)$ and~$\{v_j\}_{j\in \mathbb{N}} \subset \cS({\mathbb{R}}^n)$ be
		a sequence such that
		$$\lim_{j\to+\infty}[v-v_j]_{H_a(\mathbb{R}^n)}=0.$$
		Let~$V(x,t) := \mathcal{F}_x^{-1}(\tilde{g}(|\xi|^2,t)\hat{v})$ and~$V_j(x,t) := \mathcal{F}_x^{-1}(\tilde{g}(|\xi|^2,t)\hat{v}_j)$.
		
		Then,
		\begin{equation}\label{Convergence by density}
			\lim_{j\to+\infty}   \int_{{\mathbb{R}}^{n+1}_+}a(t)|\nabla (V_j(x,t)-V(x,t))|^2 \,dx\,dt =0.
		\end{equation}
		
		In particular, $V_j \to V$ in~$L^2_{\text{loc}}({\mathbb{R}}^{n+1}_+)$.
	\end{lemma}
	
	\begin{proof}
		Since~$v_j\to v$ in~$H_a({\mathbb{R}}^n)$, we see that
		\begin{equation*}
			\begin{split}&
				\lim_{j\to+\infty}\int_{{\mathbb{R}}^{n+1}_+} a(t) |\nabla (V_j(x,t)-V(x,t))|^2 \, dx\,dt\\
				&= \lim_{j\to+\infty}
				\int_{{\mathbb{R}}^{n+1}_+} a(t)\left(|\xi|^2 |g(|\xi|^2,t)|^2+|\partial_t g(|\xi|^2,t)|^2\right)|\hat{v}(\xi)-\hat{v}_j(\xi)|^2\, d\xi\, dt\\
				&=\lim_{j\to+\infty}\int_{{\mathbb{R}}^n} m(|\xi|^2)|\hat{v}(\xi)-\hat{v}_j(\xi)|^2\, d\xi
				\\&= 0,
			\end{split}
		\end{equation*}
		which proves~\eqref{Convergence by density}.
		
		Furthermore, by the Sobolev Theorem, we have that~$V_j \to V$ in~$L^q_{\text{loc}}(\mathbb{R}^{n+1}_+)$ for any~$q\in\left[1,\frac{2n+2}{n-1}\right)$.
	\end{proof}
	
	\subsection{Proofs of Theorem~\ref{Convolution enters the chat} and Corollary~\ref{Rigidity result}}
	
	With the work done so far, we can now complete the proofs of Theorem~\ref{Convolution enters the chat} and Corollary~\ref{Rigidity result}.
	
	\begin{proof}[Proof of Theorem~\ref{Convolution enters the chat}]
		First, assume that~$u \in \mathcal{S}(\mathbb{R}^n)$ and let~$U$ be as in~\eqref{Sol of Prob with Neumann}. {F}rom~$(P_1)$ in Lemma~\ref{Gradient of U} we see that, for all~$t\in\mathbb{R}_+$,
		\begin{eqnarray*}
			\div(a(t)\nabla U)
			&=&\div\Big(a(t)
			\big(\mathcal{F}_x^{-1}(i\xi {g}(|\xi|^2,t)\hat{u}), \mathcal{F}_x^{-1}(\partial_t {g}(|\xi|^2,t)\hat{u})\big)\Big)\\
			&=&\mathcal{F}_{x}^{-1}\left(\left(-|\xi|^2 a(t) {g}(|\xi|^2,t) + \frac{d}{dt}\left(a(t)\partial_t {g}(|\xi|^2,t)\right)\right)\hat{u}\right).
		\end{eqnarray*}
		Consequently, 
		from the equation for~$g$ in Theorem~\ref{Minimization Problem} we deduce that 
		$$
		\div(a(t)\nabla U) = 0.
		$$
		Therefore, for any~$\varphi \in C^{\infty}_c(\mathbb{R}^{n+1})$,
		using~$(P_6)$ of Lemma~\ref{Asymptotic Properties},
		\begin{equation*}
			\begin{split}
				\int_{\mathbb{R}^{n+1}_+}a(t)\nabla U\cdot \nabla \varphi\,dx\,dt &= \int_{\mathbb{R}^{n+1}_+}\div(a(t) \varphi \nabla U)-\varphi\div(a(t)\nabla U)\,dx\,dt\\
				&= -\int_{\mathbb{R}^n}\lim_{t\to 0^+} a(t)\partial_t U(x,t)\varphi \,dx\\
				&=\int_{\mathbb{R}^n}
				\mathcal{F}^{-1}(m(|\xi|^2)\hat{u})\varphi \,dx\\
				&= \int_{\mathbb{R}^n} f \varphi\,dx,
			\end{split}
		\end{equation*}
		where in the last equality
		we have used the notation in~\eqref{r43308kfjvbnmghjksagfkuegwut}.
		
		Therefore, $U$ is a weak solution of~\eqref{Prob with Neumann} and the fact that~$U$ has finite energy is a consequence of the fact that~$u \in H_a(\mathbb{R}^n)$ and Lemma~\ref{Gradient of U}. 
		
		For a general~$u \in H_a(\mathbb{R}^n)$,
		in light of Lemma~\ref{Density of Schwartz}, we can consider~$u_j \in \mathcal{S}(\mathbb{R}^n)$ such that
		\begin{equation}\label{cvgdftsd894djx90}
			\lim_{j\to+\infty}[u_j-u]_{H_a(\mathbb{R}^n)}=0.\end{equation}
		We define~$U_j := \mathcal{F}^{-1}_{x}\left(\tilde{g}(|\xi|^2,t)\hat{u}_j\right)$ and~$f_j := \mathcal{F}_{x}^{-1}\left(m(|\xi|^2)\hat{u}_j\right)$ (note that~$ f_j \in \text{Im}(\mathcal{L}_a)\cap L^1_{\text{loc}}(\mathbb{R}^n)$). 
		
		{F}rom the first part of this proof, we have that
		\begin{equation}\label{cvgdftsd894djx9023}
			\int_{\mathbb{R}^{n+1}_+}a(t) \nabla U_j \cdot \nabla \psi \,dx\,dt=
			\int_{\mathbb{R}^n}f_j \psi(x,0)\, dx.
		\end{equation}
		
		We claim that
		\begin{equation}\label{cvgdftsd894djx902}
			\lim_{j\to+\infty}\int_{\mathbb{R}^n} f_j(x)\psi(x,0)\,dx=\int_{\mathbb{R}^n} f(x) \psi(x,0)\,dx.
		\end{equation}
		Indeed, for every~$\psi \in C^{\infty}_c (\mathbb{R}^{n+1})$, 
		\begin{eqnarray*}&&
			\int_{\mathbb{R}^n} \big(f_j(x)-f(x)\big) \psi(x,0)\,dx\\&=&
			\int_{\mathbb{R}^n} \big(\mathcal{F}_{x}^{-1}\left(m(|\xi|^2)\hat{u}_j\right)(x)
			-\mathcal{F}^{-1}(m(|\xi|^2)\hat{u})(x)\big) \psi(x,0)\,dx \\&=&
			\int_{\mathbb{R}^n} m(|\xi|^2)\big(\hat{u}_j(\xi)-\hat{u}(\xi)\big)\mathcal{F}_x\psi(\xi,0)\,d\xi.
		\end{eqnarray*}
		Hence, by the H\"older inequality,
		\begin{eqnarray*}
			&&\left|\int_{\mathbb{R}^n} \big(f_j(x)-f(x)\big) \psi(x,0)\,dx\right|\\
			&\leq&\left(
			\int_{\mathbb{R}^n} m(|\xi|^2)\big|\hat{u}_j(\xi)-\hat{u}(\xi)\big|^2\,d\xi\right)^{1/2}
			\left(\int_{\mathbb{R}^n} m(|\xi|^2)|\mathcal{F}_x\psi(\xi,0)|^2\,d\xi\right)^{1/2}\\&\leq& (m(1))^{1/2}
			[{u}_j-{u}]_{H_a(\mathbb{R}^n)}
			\left(\int_{\mathbb{R}^n} (1+|\xi|^2)|\mathcal{F}_x\psi(\xi,0)|^2\,d\xi\right)^{1/2}.
		\end{eqnarray*}
		Since~$\psi\in C^{\infty}_c (\mathbb{R}^{n+1})$, from this and~\eqref{cvgdftsd894djx90}
		we deduce~\eqref{cvgdftsd894djx902}.
		
		Moreover, by Lemma~\ref{Convergence of the gradients},
		\begin{equation*}\begin{split}&\lim_{j\to+\infty}\left|
				\int_{\mathbb{R}^{n+1}_+}a(t) \big(\nabla U -\nabla U_j\big)\cdot \nabla \psi \,dx\,dt\right| \\ \leq \;&\lim_{j\to+\infty}\left(
				\int_{\mathbb{R}^{n+1}_+}a(t) \big|\nabla U -\nabla U_j\big|^2 \,dx\,dt\right)^{1/2}
				\left(\int_{\mathbb{R}^{n+1}_+}a(t) | \nabla \psi|^2 \,dx\,dt\right)^{1/2}
				\\ =\;&0.\end{split}
		\end{equation*}
		{F}rom this, \eqref{cvgdftsd894djx9023} and~\eqref{cvgdftsd894djx902}, we deduce that
		\begin{equation*}
			\begin{split}
				\int_{\mathbb{R}^{n+1}_+}a(t) \nabla U \cdot \nabla \psi \,dx\,dt &=\displaystyle\lim_{j\to \infty} \int_{\mathbb{R}^{n+1}_+}a(t) \nabla U_j \cdot \nabla \psi \,dx\,dt\\
				&=\displaystyle\lim_{j\to \infty}\int_{\mathbb{R}^n}f_j \psi(x,0)\, dx\\
				&=\int_{\mathbb{R}^n}f\psi(x,0)\, dx.
			\end{split}
		\end{equation*}
		Hence, $U$ is a weak solution of~\eqref{Prob with Neumann}. 
		
		Since we are only considering finite-energy solutions of~\eqref{Prob with Neumann}, the result follows from the uniqueness statement in Lemma~\ref{Uniqueness of finite energy}. 
	\end{proof}
	
	\begin{proof}[Proof of Corollary~\ref{Rigidity result}]
		By Theorem~\ref{Convolution enters the chat} we know that we can write~$U = \mathcal{F}_x^{-1}\left(\tilde{g}(|\xi|^2,t) \hat{u}\right)+ c$, for some constant~$c \in \mathbb{R}$. 
		
		Since~$u$ is one-dimensional, due to the assumption that~$u \in W^{1,1}_{\text{loc}}(\mathbb{R}^n)$, there exists~$\nu\in \mathbb{R}^n$ such that~$\nu\cdot \nabla u = 0$.
		
		Also, by~$(P_1)$ in Lemma~\ref{Gradient of U},
		\begin{equation*}
			(\nu,0)\cdot \nabla U =
			\nu\cdot \mathcal{F}_x^{-1}\left(i\xi\tilde{g}(|\xi|^2,t) \hat{u}\right)= 
			\mathcal{F}_x^{-1}\left(\tilde{g}(|\xi|^2,t) \hat{\nu\cdot \nabla u}\right) =0,
		\end{equation*}  
		which proves that~$U(\cdot, t):\mathbb{R}^n\rightarrow \mathbb{R}$ is also one-dimensional.
	\end{proof}
	
	\begin{appendix}
		
		\section{Proof of some technical lemmata}\label{gyu4io3vbcnx76859430}
		
		Here we provide the proofs of some technical results stated in Section~\ref{section4}.
		
		\begin{proof}[Proof of Lemma~\ref{Muckenhoupt is not L1}]
			Given an interval~$I\subset\mathbb{R}$, we denote by~$|I|$ the Lebesgue measure and, 
			for simplicity, we also use the notation
			$$ \mathcal{I}(I) := \int_I \omega(\tau)\,d\tau.$$
			
			Since~$\omega$ is an~$A_2$-Muckenhoupt weight, we can employ
			Proposition~7.2.8 in~\cite{MR3243734} (with~$p:=2$).
			In this way,
			we find that there exist~$C>0$, depending only on~$\omega$,
			and~$\delta \in (0,1)$ such that, for every interval~$I \subset \mathbb{R}_+$ and every measurable subset~$J$ of~$I$,
			\begin{equation}\label{estimate on quotient of measures}
				\frac{\mathcal{I}(J)}{\mathcal{I}(I)}\leq C \left(\frac{|J|}{|I|}\right)^{\delta}.
			\end{equation}
			
			Now, we assume, with a view by contradiction, that~$\omega \in L^1(\mathbb{R}_+)$
			and we define, for all~$j\in\mathbb{N}$,
			$$
			I_j := (0,j)\qquad\text{and}\qquad J_j = (0,j^{1/2}).
			$$
			Since~$\omega \in L^1(\mathbb{R}_+)$ we have that for every~$\eps >0$ there exists~$j_\eps \in \mathbb{N}$ sufficiently large such that
			\begin{eqnarray*}
				&&\|\omega\|_{L^{1}(\mathbb{R}_+)}-\eps \leq\mathcal{I}(I_{j_\eps}) \leq \|\omega\|_{L^{1}(\mathbb{R}_+)}\\
				{\mbox{and }}&& \|\omega\|_{L^{1}(\mathbb{R}_+)}-\eps \leq\mathcal{I}(J_{j_\eps}) \leq \|\omega\|_{L^{1}(\mathbb{R}_+)}
				.
			\end{eqnarray*}
			{F}rom this, we deduce that
			$$
			\frac{\mathcal{I}(J_{j_\eps})}{\mathcal{I}(I_{j_\eps})}
			\geq \frac{\|\omega\|_{L^1(\mathbb{R}_+)}-\eps}{\|\omega\|_{L^1(\mathbb{R}_+)}}
			=1-\frac{\eps}{\|\omega\|_{L^1(\mathbb{R}_+)}}.
			$$ 
			Therefore, using~\eqref{estimate on quotient of measures} with~$I:=I_{j_\eps}$
			and~$J:=J_{j_\eps}$, we obtain that
			\begin{equation*}
				1-\frac{\eps}{\|\omega\|_{L^1(\mathbb{R}_+)}}
				\leq \frac{\mathcal{I}(J_{j_\eps})}{\mathcal{I}(I_{j_\eps})}
				\leq C \left(\frac{|J_{j_\eps}|}{|I_{j_\eps}|}\right)^{\delta}= C j_\eps^{-\frac{\delta}{2}},
			\end{equation*}
			which, sending~$\eps\searrow0$, leads to a contradiction.
		\end{proof}
		
		\begin{proof}[Proof of Lemma~\ref{Growth bound on Muckenhoupt}]
		Let~$t_0>0$.
			Leveraging the fact that~$\omega>0$, we see that there exists~$C_{0}>0$ such that
			\begin{equation}\label{yurweir349856t7uew65473890wqopdihjs}
				\int_{0}^{t_{0}}\omega(\tau)\,d\tau\ge C_0.\end{equation}
			
			Now, for all~$t\in[0, t_0]$,
			\begin{equation*}
				\left(\int_{0}^{t_{0}}\omega(\tau)\,d\tau\right)\left(
				\int_{0}^{t}\frac{d\tau}{\omega(\tau)}\right)\leq 
				\left(\int_{0}^{t_{0}}\omega(\tau)\,d\tau\right)\left(
				\int_{0}^{t_0}\frac{d\tau}{\omega(\tau)}\right)\le Ct_0^2.
			\end{equation*}
			This and~\eqref{yurweir349856t7uew65473890wqopdihjs} give that, if~$t\in[0, t_0]$,
			\begin{equation}\label{mnbvcx98765412345678zxcvbn}
				\int_{0}^{t}\frac{d\tau}{\omega(\tau)}\le C,
			\end{equation} up to renaming~$C$.
			
			Also, for~$t\geq t_0$,
			\begin{equation*}
				\left(\int_{0}^{t_{0}}\omega(\tau)\,d\tau\right)\left(\int_{0}^{t}\frac{d\tau}{\omega(\tau)}\right)\leq \left(\int_{0}^{t}\omega(\tau)\,d\tau\right)\left(\int_{0}^{t}\frac{d\tau}{\omega(\tau)}\right)\leq Ct^2.
			\end{equation*}
			Using~\eqref{yurweir349856t7uew65473890wqopdihjs}, 
			we thus obtain that
			$$
			\int_{0}^{t}\frac{d\tau}{\omega(\tau)}\leq C t^2.
			$$
			{F}rom this and~\eqref{mnbvcx98765412345678zxcvbn}
			the desired result follows.
		\end{proof}
		
		\begin{proof}[Proof of Lemma~\ref{Density of Schwartz}]
			We first check that~${\mathcal{S}}(\mathbb{R}^n)\subset H_a(\mathbb{R}^n)$. For this,
			let~$\varphi\in{\mathcal{S}}(\mathbb{R}^n)$ and
			notice that~$\varphi\in{\mathcal{S}}'(\mathbb{R}^n)$
			and~$\hat\varphi$ is a measurable function (in fact, it belongs to~${\mathcal{S}}(\mathbb{R}^n)$). 
			
			Hence, to check that~$\varphi\in H_a(\mathbb{R}^n)$ it remains to show that~$[\varphi]_{H_a(\mathbb{R}^n)}<+\infty$. To this end, 
			we observe that~$m(|\xi|^2)\leq m(1)(1+|\xi|^2)$, and therefore
			\begin{equation*} [\varphi]_{H_a(\mathbb{R}^n)}^2=
				\int_{\mathbb{R}^n} m(|\xi|^2) |\hat{\varphi}(\xi)|^2\,d\xi
				\leq  m(1)\int_{\mathbb{R}^n} (1+|\xi|^2) |\hat{\varphi}(\xi)|^2\,d\xi
				<+\infty, 
			\end{equation*}
			as desired.
			
			We now show~\eqref{4358rufinittesalimit}.
			For this, let~$v \in H_a(\mathbb{R}^n)$ and~$\delta>0$.
			Also, pick~$\varepsilon\in(0,1)$ small enough and~$R>1$ large enough such that 
			\begin{equation*}
				\int_{B_{\varepsilon}}m(|\xi|^2)|\hat{v}(\xi)|^2\, d\xi < \frac{\delta}{3},\qquad
				\text{and}\qquad \int_{\mathbb{R}^n\setminus B_{R}}m(|\xi|^2)|\hat{v}(\xi)|^2 \, d\xi < \frac{\delta}{3}.
			\end{equation*}
			Moreover, we observe that, since~$m$ is monotone increasing,
			\begin{equation*}
				m(\varepsilon^2)\int_{B_R\setminus B_{\varepsilon}} |\hat{v}(\xi)|^2\,d\xi
				\leq \int_{B_R\setminus B_{\varepsilon}} m(|\xi|^2)|\hat{v}(\xi)|^2\,d\xi \leq m(R^2)\int_{B_R\setminus B_{\varepsilon}} |\hat{v}(\xi)|^2\,d\xi.
			\end{equation*}
			In particular, this entails that~$\hat{v} \in L^2(B_R\setminus B_{\varepsilon})$.
			
			Accordingly, we can take~$\varphi \in \mathcal{S}(\mathbb{R}^n)$ such that~$\hat{\varphi} \in C^{\infty}_c(B_R\setminus B_{\varepsilon})$ and 
			\begin{equation*}
				\int_{B_R\setminus B_{\varepsilon}} |\hat{v}(\xi)-\hat{\varphi}(\xi)|^2\,d\xi< \frac{\delta}{3m(R^2)}.
			\end{equation*}
			Then, gathering all these pieces of information, we see that
			\begin{eqnarray*}
				&&[v-\varphi]_{H_a(\mathbb{R}^n)}^2
				\\&=&\int_{\mathbb{R}^n}m(|\xi|^2)|\hat{v}(\xi)-\hat{\varphi}(\xi)|^2\,d\xi
				\\&=&
				\int_{B_\varepsilon}m(|\xi|^2)|\hat{v}(\xi)|^2\,d\xi
				+\int_{B_R\setminus B_\varepsilon}m(|\xi|^2)|\hat{v}(\xi)-\hat{\varphi}(\xi)|^2\,d\xi
				+\int_{\mathbb{R}^n\setminus B_R}m(|\xi|^2)|\hat{v}(\xi)|^2\,d\xi\\
				&\leq&\int_{B_\varepsilon}m(|\xi|^2)|\hat{v}(\xi)|^2\,d\xi +
				m(R^2)\int_{B_R\setminus B_\varepsilon}|\hat{v}(\xi)-\hat{\varphi}(\xi)|^2\,d\xi +
				\int_{\mathbb{R}^n\setminus B_R}m(|\xi|^2)|\hat{v}(\xi)|^2\,d\xi
				\\&<&\delta.
			\end{eqnarray*}
			This gives~\eqref{4358rufinittesalimit}, as desired.
		\end{proof}
	\end{appendix}
	
	\section*{Acknowledgements} 
	
	Serena Dipierro, Giorgio Poggesi, and Enrico Valdinoci are members of the Australian Mathematical Society (AustMS).
	Serena Dipierro is supported by the Australian Research Council
	Future Fellowship FT230100333
	``New perspectives on nonlocal equations''.
	Jo{\~a}o Gon\c{c}alves da Silva and
	Giorgio Poggesi are supported by the Australian Research Council Discovery Early Career Researcher Award (DECRA) DE230100954 ``Partial Differential Equations: geometric aspects and applications''.
	Jo{\~a}o Gon\c{c}alves da Silva is supported by a Scholarship for International Research Fees at The University of Western Australia.
	Enrico Valdinoci is supported by the Australian Laureate Fellowship FL190100081 ``Minimal surfaces, free boundaries and partial differential equations''.

\end{document}